\documentclass[12pt]{amsart}
\usepackage{amsmath}
\usepackage{amsfonts}
\usepackage{amsthm}
\usepackage{amssymb}
\usepackage{amscd}
\usepackage[all]{xy}
\usepackage{enumerate}
\usepackage{hyperref}
\usepackage{etoolbox}
\usepackage{dynkin-diagrams}
\usepackage{multirow}

\makeatletter
\@namedef{subjclassname@2020}{%
  \textup{2020} Mathematics Subject Classification}
\makeatother



\theoremstyle{plain}
\newtheorem{thm}{Theorem}[section]
\newtheorem{thml}{Theorem}

\newtheorem{prop}[thm]{Proposition}

\newtheorem{cor}[thm]{Corollary}

\newtheorem{lem}[thm]{Lemma}

\theoremstyle{definition}
\newtheorem{defn}[thm]{Definition}

\newtheorem{prob}[thm]{Problem}

\newtheorem{rmk}[thm]{Remark}
\newcommand{\lieg}{\mathfrak{g}}
\newcommand{\lieh}{\mathfrak{h}}
\newcommand{\lien}{\mathfrak{n}}

\newcommand{\sB}{\mathcal{B}}
\newcommand{\sC}{\mathcal{C}}

\newcommand{\sE}{\mathcal{E}}
\newcommand{\sF}{\mathcal{F}}
\newcommand{\sG}{\mathcal{G}}

\newcommand{\sI}{\mathcal{I}}
\newcommand{\sJ}{\mathcal{J}}

\newcommand{\sL}{\mathcal{L}}

\newcommand{\sO}{\mathcal{O}}
\newcommand{\sP}{\mathcal{P}}

\newcommand{\sX}{\mathcal{X}}

\newcommand{\mC}{\mathbb{C}}

\newcommand{\mP}{\mathbb{P}}

\newcommand{\mZ}{\mathbb{Z}}
\newcommand{\mR}{\mathbb{R}}

\newcommand{\Ima}{\mathrm{Im}\,}

\newcommand{\e}{\varepsilon}

\newcommand{\pic}{\mathrm{Pic}}
\usepackage{color}
\numberwithin{equation}{section}

\newcommand{\beba}  {\begin{equation}\begin{array}{rcl}}

\newcommand{\eaee}  {\end{array}\end{equation}}

\pgfkeys{/Dynkin diagram,
edge length=1.3cm,
fold radius=.5cm,
indefinite edge/.style={
draw=black,
fill=white,
thin,
densely dashed}}

\makeatletter
\def\l@section{\@tocline{1}{0pt}{1pc}{}{}}
\def\l@subsection{\@tocline{2}{0pt}{1pc}{4.6em}{}}
\def\l@subsubsection{\@tocline{3}{0pt}{1pc}{7.6em}{}}
\renewcommand{\tocsection}[3]{%
  \indentlabel{\@ifnotempty{#2}{\makebox[2.3em][l]{%
    \ignorespaces#1 #2.\hfill}}}#3}
\renewcommand{\tocsubsection}[3]{%
  \indentlabel{\@ifnotempty{#2}{\hspace*{2.3em}\makebox[2.3em][l]{%
    \ignorespaces#1 #2.\hfill}}}#3}
\renewcommand{\tocsubsubsection}[3]{%
  \indentlabel{\@ifnotempty{#2}{\hspace*{4.6em}\makebox[3em][l]{%
    \ignorespaces#1 #2.\hfill}}}#3}
\makeatother

\title{On Massey products and Rational homogeneous varieties}

\author{Luca Rizzi}
\address{L. Rizzi: JSPS International Fellow, Graduate School of Mathematical Sciences, the University of Tokyo,
Tokyo, 153-8914 Japan;
\texttt{rizzil@ms.u-tokyo.ac.jp}}

\usepackage[foot]{amsaddr}

\begin{document}

\markboth{}{}
\begin{abstract} 
We study the equivalence between infinitesimal Torelli theorem for smooth hypersurfaces in rational homogeneous varieties with Picard number one and the theory of generalized Massey products.
This equivalence shows that the differential of the period map vanishes on an infinitesimal deformation if and only if certain twisted differential forms are elements of the Jacobian ideal of the hypersurface.
We also prove an infinitesimal Torelli theorem result for smooth hypersurfaces in log parallelizable varieties.
\end{abstract}
\maketitle
%\tableofcontents
%%%%%%%%%%%%%%%%%%%%%%%%%%%%%%%%%%%%%%%%%%%%%%%%%%%%%%%
%%%%%%%%%%%%%%%%%%%%%%%%%%%%%%%%%%%%%%%%%%%%%%%%%%%%%%%
%%%%%%%%%%%%%%%%%%%%%%%%%%%%%%%%%%%%%%%%%%%%%%%%%%%%%%%

\section{Introduction}

Let $V\subset \mP^n$ be a smooth hypersurface defined by a homogeneous polynomial $F$ of degree $d$. Denote by $S$ the polynomial ring in $n+1$ variables and by $\sJ$ the Jacobian ideal of $F$, that is the ideal of $S$ generated by the partial derivatives of $F$. The Jacobian ring is by definition the quotient $R = S/\sJ$ and it is naturally graded. 

Griffiths in \cite{Gri1} showed that the graded pieces of the ring $R$ give all the pieces of the Hodge structure of $V$. Furthermore, the piece of degree $d$ gives exactly the space of infinitesimal deformations of $V$. Using this results he proved that the infinitesimal Torelli problem for projective hypersurfaces is reduced to the study of the injectivity of suitable maps given by polynomial multiplication.

This idea was later developed by Green in \cite{green1} to deal with the general case of smooth sufficiently ample hypersurfaces of a smooth projective variety. Green introduced the notions of pseudo-Jacobian ideal and of pseudo-Jacobian ring and developed an approach to the problem extending the one by Griffiths. In this paper Green was able to solve the infinitesimal and generic Torelli problem for smooth sufficiently ample hypersurfaces.
  
Now take a smooth projective variety $Y$ with $\pic(Y)=[H]\cdot \mathbb Z$ where $H$ is an effective divisor. Hence for any effective divisor $X\subset Y$ there exists a unique $d\in\mathbb N_{\geq 0}$ such that $X$ is an element of the linear system $|dH|$. Thus a natural problem is to concretely specify the ampleness condition required by Green and find the minimum $d\in\mathbb N$ such that the differential of the period map $d\sP_{X}$ is injective for $X$ a smooth element of $|dH|$. 

In \cite{RZ2} and \cite{RZ3} an interesting equivalence between this problem and the theory of \emph{Massey products} is developed in the case of $Y$ a projective space and $Y$ a Grassmannian variety respectively. In this paper we generalize these results for the case $Y$ rational homogeneous variety of Picard number one.

Let us briefly recall the notion of generalized Massey product, more details of the general theory are discussed in Section \ref{aggiunta}; see also the therein quoted bibliography. 
Denote by $\sO_{Y}(1)$ the ample generator of $\pic(Y)$ and by $\sO_{Y}(a)=\sO_{Y}(1)^{\otimes a}$ its twist by an integer $a$ and take an infinitesimal deformation $\xi\in H^1(X,T_X)$ of $X$. If $X=(\tau=0)$, $\tau\in H^{0}(Y,\sO_{Y}(d))$, any infinitesimal deformation is induced by a local family $X+tR=0$ where $R\in H^{0}(Y,\sO_{Y}(d))$; see: Proposition \ref{etutto}.
We twist by $\sO_X(2)$ the exact sequence associated to the infinitesimal deformation $\xi\in H^1(X,T_X)$ to obtain
\begin{equation}
\label{estensionedivisoreduedue}
0\to\sO_X(2)\to \Omega^1_{\sX}(2)|_X\to\Omega^1_X(2)\to0.
\end{equation}
Now take $n=\dim X$ and a generic $n+1$-dimensional vector space $W<H^0(X,\Omega^1_X(2))$ contained in the kernel of the cup-product homomorphism $\delta_{\xi}\colon H^0(X,\Omega^1_X(2))\to H^1(X,\sO_X(2))$ and denote by $\lambda^{i}W$ the image of $\bigwedge^iW$ through the natural homomorphism $\lambda^i\colon\bigwedge^i H^0(X,\Omega^1_X(2))\to  H^0(X,\bigwedge^{i}(\Omega^1_X(2)))$.

Let $\sB:=\langle\eta_{1},\ldots,\eta_{n+1}\rangle$ be a basis of $W$ and $s_1,\ldots, s_{n+1}\in H^0(X, \Omega^1_{\sX}(2)|_X)$ liftings of, respectively,  $\eta_{1},\ldots,\eta_{n+1}$, then the map 
$$\Lambda^{n+1}\colon\bigwedge^{n+1} H^0(X, \Omega^1_{\sX}(2)|_X)\to H^0(X,  {\rm{det}}(\Omega^1_{\sX}(2)|_X))$$ gives the twisted differential form $ \Omega:=\Lambda^{n+1}(s_1\wedge s_2 \wedge\ldots\wedge s_{n+1})\in H^0(X,{\rm{det}}( \Omega^1_{\sX}(2)|_X))$ which is called {\it{generalized adjoint form or Massey product associated to}} $\xi$, $W$, and $\sB$. On the other hand we can also consider the $n+1$ global forms $\omega_i:=\lambda^{n}(\eta_1\wedge\ldots\wedge\eta_{i-1}\wedge{\widehat{\eta_i}}\wedge\eta_{i+1}\wedge\ldots\wedge\eta_{n+1})\in \lambda^nW$.

By Sequence (\ref{estensionedivisoreduedue}), 
 $\det(\Omega^1_{X}(2))\otimes\sO_{X}(2)=\det (\Omega^1_{\sX}(2)|_X)$, and we can construct an obvious homomorphism:
\begin{equation}
\label{aggiuntasequenzaaa}
H^0(X,\sO_X(2))\otimes \lambda^nW \to H^0(X, {\rm{det}}(\Omega^1_{\sX}(2)|_X).
\end{equation} 

The Generalized Adjoint Theorem, see \ref{theoremA}, fully characterizes  the condition of $\Omega$ being in the image of this map, that is 
 $$\Omega \in \Ima(H^0(X,\sO_X(2)\otimes \lambda^nW \to H^0(X,{\rm{det}}(\Omega^1_{\sX}(2)|_X).$$
The important point is that to check this condition is equivalent to check if $d\sP_{X}(\xi)=0$. 

More deeply, we can construct an explicit space of generalized adjoint forms associated to $\xi$ in the following way.
First we lift the sections $\eta_i$'s from $H^0(X,\Omega^1_X(2))$ to $H^0(Y,\Omega^1_Y(2))$, then we take the wedge product to obtain a twisted volume form $\widetilde{\Omega}$, independent from $\xi$, such that 
$\widetilde{\Omega}\in H^0(Y,\Omega_Y^N(2N))$, where $N=\dim Y$. Finally by taking the form $R\in H^0(Y,\sO_Y(d))$ which induces the deformation $\xi$, we consider the class $R\cdot{\widetilde{\Omega}}\in  H^0(Y,\Omega_Y^N(2N+d))$ which by adjunction restricts to $\Omega\in H^0(X,\Omega^{N-1}_X(2N))=H^0(X,{\rm{det}}(\Omega^1_{X}(2))\otimes\sO_{X}(2))$. 
The following theorem highlights the aforementioned equivalence between infinitesimal Torelli problem and Massey products.
\begin{thml}
\label{A}
Let $Y$ one of the rational homogeneous varieties in Table \ref{tabellapesi} and $X=(\tau=0)$ a smooth hypersurface of degree $d\geq3$, except the cases $Y=Q^3$, $d=3$ and $Y=(D_3,\alpha_2)$, $d=3$. The following are equivalent
\begin{itemize}
	\item[\textit{i})] the differential of the period map $d\sP_X$ is zero on the infinitesimal deformation $\xi$ induced by $R\in H^0(Y,\sO_Y(d))$;
	\item[\textit{ii})] $R$ is an element of the pseudo-Jacobi ideal $\sJ_{d}$;
	\item[\textit{iii})]  for a generic $\xi$-Massey product $\Omega$ it holds $$\Omega\in \Ima H^0(X,\sO_X(2))\otimes \lambda^nW\to H^0(X,\Omega^{N-1}_X(2N));$$
	\item[\textit{iv})] if $\widetilde{\Omega}\in H^0(Y,\Omega_Y^N(2N))$ restricts to a generalized Massey product then $R\widetilde{\Omega}\in \sJ_{2N-k(Y)+d}$.
\end{itemize}
\end{thml} See Theorem \ref{main}.
We point out that this results extends the one of \cite{RZ3} and even in the case of Grassmannians provides a better bound for the degree $d$ of the hypersurface $X$.
As in \cite{RZ2}, \cite{RZ3}, we focus on hypersurfaces because the infinitesimal Torelli is known to fail if the codimension is $\geq2$, see \cite{CZ}.

Since in \cite{RZ2} we study the case of smooth hypersurfaces in projective spaces, another interesting generalization that can be undertaken is the study of smooth hypersurfaces in toric varieties and, more in general, of smooth hypersurfaces in log parallelizable varieties.
This is what we do in Section \ref{par}.

Recall that a log parallelizable variety is given by a pair $(X,D)$ where $X$ is a smooth variety and $D$ is a reduced normal crossing divisor in $X$ such that the logarithmic tangent sheaf $T_X(-\textnormal{log } D)$ is the direct sum of copies of $\sO_X$.

In this section we start by fixing some vanishing results, in particular the analogue of the vanishing theorem of Bott-Steenbrink-Danilov \cite[Theorem 7.1]{cox}. 
Then we prove the infinitesimal Torelli theorem for certain hypersurfaces in log parallelizable varieties, see Theorem \ref{torellilog}.
\begin{thml}
\label{B}
Take an ample line bundle $L$ on $X$, $\dim X=n$, and $Z$ a smooth element in the linear system $|L|$. Assume that $\Omega^n_X\otimes L$ is ample and that 
\begin{equation}
H^0(X,\Omega^n_X\otimes L)\otimes H^0(X,\Omega^n_X\otimes L^{n-1})\to H^0(X,{\Omega^n_X}\otimes\Omega^n_X\otimes L^n)
\end{equation}
is surjective. 
Then the infinitesimal Torelli theorem holds for $Z$.
\end{thml}

This paper is organized as follows. In Section \ref{sez2} and \ref{cotangent} we review the theory of rational homogeneous varieties, in particular the Generalized Borel-Weyl Theorem \ref{bw} and we give some important results on the twisted cotangent sheaves $\Omega^p(a)$. These results are essential to apply the theory of generalized Massey products which is recalled in Section \ref{aggiunta}. In Section \ref{green} we review the notions of pseudo-Jacobian ideal and ring originally introduced by Green. Finally in Sections \ref{sez6} and \ref{sez7} we show why $2$ is the correct twist to consider in Sequence (\ref{estensionedivisoreduedue}) and we put everything together to obtain Theorem \ref{A}.
In Section \ref{par} we study the hypersurfaces of log-parallelizable varieties and we prove Theorem \ref{B}.

\section{Rational homogeneous varieties with Picard number 1}
\label{sez2}
In this section we briefly review the theory of rational homogeneous varieties while we fix the notation. For representation theory and Lie algebras we refer to \cite{FH} and \cite{Hump}, see also \cite{O} for some notes on rational homogeneous varieties. 

A classical theorem of Borel and Remmert \cite{BR} asserts that a homogeneous projective complex manifold is a direct product of an abelian variety by a rational homogeneous space. The latter can be described as a quotient $G/P$, where $G$ is a semi-simple algebraic group and $P$ a parabolic subgroup. 
Moreover such rational homogeneous space can be decomposed as
\begin{equation}
G/P=G_1/P_1\times\dots\times G_k/P_k
\label{decomp}
\end{equation} where the $G_i$ are simple algebraic groups and $P_i<G_i$ are parabolic subgroups. 

As customary we will denote by $\lieg$ the Lie algebra associated to an algebraic group $G$, that is $\lieg=T_eG$, and by $\text{ad}\colon \lieg\to \text{End}(\lieg)$ its adjoint representation, defined via the Lie bracket by $\text{ad}(x)(y)=[x,y]$. The Killing form on $\lieg$ is then $K(x,y)=\text{tr}(\text{ad}(x)\circ\text{ad}(y))$.

Let $\lieh$ be the Cartan subalgebra of $\lieg$ and recall that an element $\alpha$ in its dual $\lieh^*$ is called a \emph{root} if 
$$
\lieg_\alpha=\{x\in\lieg\mid \text{ad}(h)(x)=\alpha(h)x\quad\forall h\in \lieh\}
$$ 
is non trivial. In this case it is well known that $\lieg_\alpha$ is a 1-dimensional subspace of $\lieg$.
The set of roots is usually denoted by $\Phi$ and one has the \emph{Cartan decomposition} of $\lieg$
$$
\lieg=\lieh\oplus\bigoplus_{\alpha\in\Phi}\lieg_\alpha.
$$

Choosing an hyperplane not containing any root, we identify a decomposition $\Phi=\Phi^+\cup\Phi^-$ into positive and negative roots.
An element of $\Phi ^{+}$ is called a simple root if it cannot be written as the sum of two elements of $\Phi ^{+}$, denote by $\Delta=\{\alpha_1,\dots,\alpha_l\}$ a set of simple roots for $\lieg$.
Every root $\alpha \in \Phi$  is a linear combination of elements of $\Delta$ with integer coefficients which are either all non-negative or all non-positive.

Since the Killing form $K$ is nondegenerate, it gives an isomorphism $\lieh\cong\lieh^*$ and also induces a nondegenerate form on $\lieh^*$ which we will denote using simple parenthesis $(\alpha,\beta)$.
This form is positive definite on the real span of the roots $\lieh_\mR^*$: it is known that at most two root lengths occur in $\Phi$, hence we will speak of long roots and short roots.
We also denote by $W$ be the \emph{Weyl group of $\lieg$} generated by the reflections $\sigma_\alpha$, with respect to the hyperplane orthogonal to $\alpha\in \Phi$. The fundamental Weyl chamber $\sC$ is the convex set
$$
\sC=\{\lambda\in \lieh_\mR^*\mid (\lambda,\alpha)>0\quad \forall\alpha\in \Delta\}.
$$

Given  a representation $V$ of $\lieg$, $\lambda\in\lieh^*$ is called a weight if the subspace $V_\lambda=\{v\in V\mid h\cdot v=\lambda(h)v,\ \forall h\in\lieh\}$ is non trivial.
Weights have many important properties, in particular they form the \emph{group of weights}
$$
\Lambda=\{\lambda\in\lieh_\mR^*\mid \langle\lambda,\alpha\rangle:=2(\lambda,\alpha)/(\alpha,\alpha)\in \mZ\quad\forall\alpha\in\Phi\}.
$$ This is a lattice generated by the \emph{fundamental weights} $\lambda_1,\dots,\lambda_l$ defined by the property that $\langle\lambda_i,\alpha_j\rangle=\delta_{ij}$.
 A weight $\lambda$ is \emph{dominant} if $\langle\lambda,\alpha_i\rangle\geq 0$ for every $\alpha_i\in\Delta$, equivalently $\lambda$ is a combination with non-negative integers of the fundamental weights. The weight
\begin{equation}
\delta=\sum_{i=1}^l\lambda_i=\frac{1}{2}\sum_{\alpha\in\Phi^+}\alpha
\label{delta}
\end{equation} is dominant and satisfies $\langle\delta,\alpha_i\rangle=1$ for every $\alpha_i$.
The lattice $\Lambda$ is partially ordered: $\lambda_1>\lambda_2$ if $\lambda_1-\lambda_2$ is a sum of positive roots, hence for every representation of $\lieg$, it make sense to consider its highest and lowest weights.
If such a representation is irreducible, the highest and lowest weight are unique and they uniquely identify the representation (up to isomorphism).

Now consider again a semi-simple group $G$. Every parabolic subgroup is identified, up to conjugation, by a subset of the simple roots, $\Sigma\subset \Delta$. The cardinality of this subset $\Sigma$ determines the rank of the Picard group of $G/P$ and since in this paper we are interested in varieties with $\pic\cong\mZ$, it will be enough consider varieties $G/P$ where $G$ is a simple group and $P$ is a parabolic subgroup generated by only one simple root.
Hence we recall that all the simple Lie algebras are classically known and fall in four families $A_n, B_n, C_n$, and $D_n$ with five exceptions $E_6, E_7, E_8, F_4$, and $G_2$, see Table \ref{table1} for their Dynkin diagram and the numbering of the simple roots used in the rest of the paper.

\begin{table}
	\centering
		\begin{tabular}{ll}
		$A_l$ & \dynkin[labels={1,2,l-1,l}]A{}\\
			$B_l$ & \dynkin[labels={1,2,l-2,l-1,l}]B{}\\
			$C_l$ & \dynkin[labels={1,2,l-2,l-1,l}]C{}\\
			$D_l$ & \dynkin[labels={1,2,l-3,l-2,l-1,l}]D{}\\
			$E_6$ & \dynkin[labels={1,2,3,4,5,6}]E6\\
			$E_7$ & \dynkin[labels={1,2,3,4,5,6,7}]E7\\
			$E_8$ & \dynkin[labels={1,2,3,4,5,6,7,8}]E8\\
			$F_4$ & \dynkin[labels={1,2,3,4}]F4\\
			$G_2$ & \dynkin[reverse arrows, labels={1,2}]G2
		\end{tabular}
		\caption{Dynkin diagrams of simple Lie algebras}
		\label{table1}
\end{table}

Now the subgroup $P$ is constructed in the following way.
Take $\alpha_r$, $1\leq r\leq l$, a simple root and define these three sets:
\begin{equation}
\Phi_1=\{ \alpha\in\Phi\mid\alpha=\sum_{i=1}^ln_i\alpha_i,\ n_r=0\},
\label{set1}
\end{equation}
\begin{equation}
\Phi(\lien^+)=\{ \alpha\in\Phi^+\mid\alpha=\sum_{i=1}^ln_i\alpha_i,\ n_r>0\},
\label{set2}
\end{equation}
\begin{equation}
\Phi(\lien)=\Phi_1\cup\Phi(\lien^+)
\label{set3}
\end{equation}
From these and the Cartan decomposition we obtain three subalgebras of $\lieg$:
\begin{equation}
\lieg_1=\lieh\oplus\bigoplus_{\alpha\in\Phi_1}\lieg_\alpha,\quad \lien^+=\bigoplus_{\alpha\in\Phi(\lien^+)}\lieg_\alpha,\quad \lien=\lieh\oplus\bigoplus_{\alpha\in\Phi(\lien)}\lieg_\alpha
\label{tresub}
\end{equation}
In particular $\lien$ is the subalgebra which corresponds to the parabolic subgroup $P$, that is $T_eP=\lien$.

Since all the rational homogeneous varieties with $\pic\cong\mZ$ are constructed in this way, in this paper every variety will be associated to a pair $(\lieg,\alpha_r)$. Note that different pairs can still produce isomorphic varieties.

As an important example we have the irreducible Hermitian symmetric spaces of compact type which in turn contain well known classes like Grassmannians $(A_l,\alpha_r)$ and Quadric hypersurfaces $(B_l,\alpha_1)$ and $(D_l,\alpha_1)$.

We finish this section with some important results on homogeneous vector bundles.
Homogeneous vector bundles are constructed in the following way:
take $\rho\colon P\to \text{GL}(V)$ a representation of $P$, then the associated bundle $E_\rho$ is given by the quotient of $G\times V$ by the following action of $P$
$$
(g,v)\to (gp,\rho(p^{-1})v).
$$
We will say that $E_\rho$ is irreducible if $\rho$ is an irreducible representation.
It is known, see \cite[Page 89]{K}, that there is a one to one correspondence between irreducible representations of $P$ and irreducible representations of $\lieg_1$ with lowest weight $-\lambda$, where $\lambda$ is in the sublattice $\Lambda^+_1=\{\lambda\in \Lambda\mid (\lambda,\alpha_i)\geq0,\ \forall i\neq r\}$.
Hence every irreducible vector bundle will be denoted by $E_{-\lambda}$ to highlight its corresponding lowest weight.

\begin{defn}
A weight $\lambda\in\Lambda$ is called 
\begin{itemize}
	\item[-] \textit{singular} if $(\lambda,\alpha)=0$ for some $\alpha\in \Phi$
	\item[-] \textit{regular with index $p$} if it is not singular and there exist exactly $p$ roots $\alpha\in \Phi^+$
with $(\lambda,\alpha)<0$.
\end{itemize}
\end{defn}
The following theorem, originally proved by Bott in \cite{bott} and then generalized by Kostant in \cite{Kos}, is fundamental to compute the cohomology of irreducible homogeneous vector bundle.
\begin{thm}[Generalized Borel-Weil]
\label{bw}
Let $E_{-\lambda}$ an irreducible homogeneous vector bundle as above. Then
\begin{enumerate}
	\item if the weight $\lambda+\delta$ is singular, $H^q(G/P,E_{-\lambda})=0$ for all $q$
	\item if the weight $\lambda+\delta$ is regular of index $p$, $H^q(G/P,E_{-\lambda})=0$ for $q\neq p$ and $H^p(G/P,E_{-\lambda})$ is an irreducible $G$-module with lowest weight $-\mu$. Here $\mu$ is the uniquely determined element of the closure of the fundamental Weyl chamber such that $\mu+\delta$ is congruent to $\lambda+\delta$ under the Weyl group.
\end{enumerate}
\end{thm}
\begin{rmk}
In particular if $\lambda+\delta$ is regular of index $0$, $\lambda$ is dominant, therefore $H^0(G/P,E_{-\lambda})$ is the irreducible module of lowest weight $-\lambda$, which we will denote by $G_{-\lambda}$.
\end{rmk}

The dimension of such representations can be computed using the Weyl formula \cite[Page 139]{Hump}
\begin{equation}
\dim G_{-\lambda}=\prod_{\alpha\in \Phi^+}\frac{(\lambda+\delta,\alpha)}{(\delta,\alpha)}=\prod_{\alpha\in \Phi^+}\frac{\langle\lambda+\delta,\alpha\rangle}{\langle\delta,\alpha\rangle}.
\label{weilformula}
\end{equation}

\section{The Cotangent sheaf and its twists}
\label{cotangent}

In this section we study the cotangent sheaf $\Omega^1_{G/P}$ and its twists.
From now on we will denote $G/P$ simply by $Y$; as seen before it is given by a pair $(\lieg,\alpha_r)$. Denote by $\sO_{Y}(1)$ the ample generator of $\pic(Y)$ and by $\sO_{Y}(a)=\sO_{Y}(1)^{\otimes a}$ its twist by an integer $a$.
Note that $\sO_{Y}(1)$ is induced by the irreducible representation of $P$ with lowest weight $-\lambda_r$, hence $\sO_{Y}(a)$ has lowest weight $-a\lambda_r$. 

It is well known that the tangent sheaf of $Y$ is a homogeneous bundle and the representation of $P$ associated to this bundle comes from the adjoint representation: $P\to\text{GL}(\lieg/\lien)$. Hence we can identify the cotangent bundle with the $P$-module $\lien^+$ and, more generally, the sheaf $\Omega^p_{Y}$ of $p$-forms on $Y$ will be defined by the module $\bigwedge^p\lien^+$.

The problem with these modules is that they are not irreducible in general and therefore we can not apply directly the Generalized Borel-Weil Theorem \ref{bw} when we need to compute their cohomology. 
In the case of irreducible Hermitian symmetric spaces of compact type, the modules $\bigwedge^p\lien^+$ are at least completely reducible and their decomposition is given by Kostant in \cite[Corollary 8.2]{Kos}, so it is still possible to apply the Borel-Weil theorem to their direct summands.

In the general case however they are not even completely reducible and the idea of how to deal with them is given in \cite{K1}.
Konno in his works introduces a filtration of the modules $\bigwedge^p\lien^+$ as follows: for every positive integer $i$, define $F^i\bigwedge^p\lien^+$ the linear subspace of $\bigwedge^p\lien^+$ spanned by vectors whose weights are $\lambda=\sum_{j=1}^ln_j\alpha_j$ with $n_r\geq i$. It is known that for two roots $\alpha$ and $\beta$ 
\[ 
  [\lieg_\alpha,\lieg_\beta] = 
  \begin{dcases*} 
  \lieg_{\alpha+\beta} & if  $\alpha+\beta$ is a root \\ 
  0 & otherwise 
  \end{dcases*} 
\]
hence it is not difficult to see that $F^i\bigwedge^p\lien^+$ is invariant under the adjoint action and if we set 
$$
G^i\bigwedge^p\lien^+:=F^i\bigwedge^p\lien^+/F^{i+1}\bigwedge^p\lien^+
$$ 
this is a completely reducible $P$-module and as a $\lieg_1$ module it can be identified with the subspace of $\bigwedge^p\lien^+$ spanned by the vectors whose weight satisfies $n_r=i$.
We use the notation $G^i\Omega^p_{Y}$ to indicate the homogeneous vector bundle induced by $G^i\bigwedge^p\lien^+$.
From this filtration we get a spectral sequence with the following properties, see \cite[Proposition 2.2.2]{K}
\begin{prop}
There is a spectral sequence
\begin{equation}
E^{i,q-i}_1=H^q(Y,G^i\Omega^p_Y(a))\Rightarrow H^q(Y,\Omega^p_Y(a))
\label{spectral}
\end{equation}
where $G^i\Omega^p_Y(a)=G^i\Omega^p_Y\otimes \sO_Y(a)$.

If the irreducible decomposition of $G^i\bigwedge^p\lien^+$ is 
$$
G^i\bigwedge^p\lien^+=\bigoplus_j E_{-\mu_{ij}}
$$ with $E_{-\mu_{ij}}$ irreducible with lowest weight $-\mu_{ij}$, then $-\mu_{ij}$ is a sum of $p$ distinct roots in $\Phi(\lien^+)$ and $n_r=i$.
Moreover if $E_{-\mu_{ij}}\otimes\sO_Y(a)=E_{-(\mu_{ij}+a\lambda_r)}$ is the irreducible homogeneous vector bundle with lowest weight $-(\mu_{ij}+a\lambda_r)$ we have
\begin{equation}
H^q(Y,G^i\Omega^p_Y(a))=\bigoplus_j H^q(Y,E_{-(\mu_{ij}+a\lambda_r)}).
\end{equation}
\end{prop} 
Using this spectral sequence we can give a first estimate on the cohomology $h^0(Y,\Omega^p_Y(a))$
\begin{cor}
$
h^0(Y,\Omega^p_Y(a))\leq\sum_i h^0(Y,G^i\Omega^p_Y(a))
$
\end{cor}
However we can get a more precise result on $h^0(Y,\Omega^1_Y(a))$. In this case in fact the modules $G^i\lien^+$ are irreducible and Konno in \cite[Table 3]{K1} gives a list of their lowest weights. We recall in Table \ref{tabellapesi} this list for convenience. Note that some of the pairs $(\lieg,\alpha_r)$ are omitted due to existing isomorphisms $(\lieg,\alpha_r)\cong(\lieg',\alpha_{r'})$.

\begin{table}
\begin{center}
\begin{tabular}{ccc}
    \hline
		$\lieg$&r&Lowest weights of $G^i\lien^+$\\
		\hline
		$A_l$, $(l\geq1)$&$1\leq r\leq l$&$\alpha_r$\\
		\hline		
		\multirow{2}{*}{$B_l$, $(l\geq2)$}&1&$\alpha_1$\\
    &$2\leq r\leq l-1$&$\alpha_r,\ \lambda_r-\lambda_{r-2}$\\
    \hline
		\multirow{2}{*}{$C_l$, $(l\geq3)$}& $2\leq r\leq l-1$ & $\alpha_r,\ 2\lambda_r-2\lambda_{r-1}$\\
    & $l$ & $\alpha_l$\\
    \hline
		\multirow{3}{*}{$D_l$, $(l\geq3)$}& $1$ & $\alpha_1$\\
		& $2\leq r\leq l-2$ & $\alpha_r,\ \lambda_r-\lambda_{r-2}$\\
    & $l-1$ & $\alpha_{l-1}$\\ 
    \hline
		\multirow{4}{*}{$E_6$}& $1$ & $\alpha_1$\\
		& $2$ & $\alpha_2,\ \lambda_2$\\
    & $3$ & $\alpha_3,\ \lambda_3-\lambda_6$\\
		&$4$ & $\alpha_4,\ \lambda_4-\lambda_1-\lambda_6,\ \lambda_4-\lambda_2$\\
    \hline
		\multirow{7}{*}{$E_7$}& $1$ & $\alpha_1,\ \lambda_1$\\
		& $2$ & $\alpha_2,\ \lambda_2-\lambda_7$\\
    & $3$ & $\alpha_3,\ \lambda_3-\lambda_6,\ \lambda_3-\lambda_1$\\
		&$4$ & $\alpha_4,\ \lambda_4-\lambda_1-\lambda_6,\ \lambda_4-\lambda_2-\lambda_7,\ \lambda_4-\lambda_3$\\
		&$5$&$\alpha_5,\ \lambda_5-\lambda_1-\lambda_7,\ \lambda_5-\lambda_2$\\
		&$6$&$\alpha_6,\ \lambda_6-\lambda_1$\\
		&$7$&$\alpha_7$\\
    \hline	
		\multirow{8}{*}{$E_8$}& $1$ & $\alpha_1,\ \lambda_1-\lambda_8$\\
		& $2$ & $\alpha_2,\ \lambda_2-\lambda_7,\ \lambda_2-\lambda_1$\\
    & $3$ & $\alpha_3,\ \lambda_3-\lambda_6,\ \lambda_3-\lambda_1-\lambda_8,\ \lambda_3-\lambda_2$\\
		&$4$ & $\alpha_4,\ \lambda_4-\lambda_1-\lambda_6,\ \lambda_4-\lambda_2-\lambda_7,\ \lambda_4-\lambda_3-\lambda_8,\ \lambda_4-\lambda_1-\lambda_2,\ \lambda_4-\lambda_6$\\
		&$5$&$\alpha_5,\ \lambda_5-\lambda_1-\lambda_7,\ \lambda_5-\lambda_2-\lambda_8,\ \lambda_5-\lambda_3,\ \lambda_5-\lambda_6$\\
		&$6$&$\alpha_6,\ \lambda_6-\lambda_1-\lambda_8,\ \lambda_6-\lambda_2,\ \lambda_6-\lambda_7$\\
		&$7$&$\alpha_7,\ \lambda_7-\lambda_1,\ \lambda_7-\lambda_8$\\
		&$8$&$\alpha_8,\ \lambda_8$\\
    \hline		
		\multirow{4}{*}{$F_4$}& $1$ & $\alpha_1,\ \lambda_1$\\
		& $2$ & $\alpha_2,\ \lambda_2-2\lambda_4,\ \lambda_2-\lambda_1$\\
    & $3$ & $\alpha_3,\ 2\lambda_3-\lambda_1-2\lambda_4,\ \lambda_3-\lambda_4,\ 2\lambda_3-\lambda_2$\\
		&$4$ & $\alpha_4,\ 2\lambda_4-\lambda_1$\\
    \hline	
		$G_2$&$2$&$\alpha_2,\ \lambda_2$\\
		\hline
\end{tabular}
\end{center}
\caption{Lowest Weights of $G^i\lien^+$}
\label{tabellapesi}
\end{table} 

\begin{rmk}
Note that the weights listed in Table \ref{tabellapesi} are of two kinds:
\begin{enumerate}
	\item $G^1\lien^+$ has always lowest weight $\alpha_r$
	\item $G^i\lien^+$, $i\geq2$, has always lowest weight $\lambda_r-\lambda'$ or $2\lambda_r-\lambda'$ where $\lambda'$ is a sum of fundamental weights which does not contain $\lambda_r$.
\end{enumerate}
\end{rmk}
Since we want to apply the Generalized Borel Weil to the twists $G^i\Omega_Y^1(a)$, $a\geq1$, all the weights of Table \ref{tabellapesi} will be changed by $-a\lambda_r$ and we get:
\begin{lem}
\label{studiopesi}
The weights 
\begin{enumerate}
	\item $-\alpha_r+a\lambda_r+\delta$ are singular if $a=1$ and regular of index 0 if $a>1$
	\item $-\lambda_r+\lambda'+a\lambda_r+\delta$ are regular of index 0
	\item $-2\lambda_r+\lambda'+a\lambda_r+\delta$ are singular if $a=1$ and regular of index 0 if $a>1$
\end{enumerate}
\end{lem}
\begin{proof}
This is a simple computation on the bilinear form $(,)$. Recall that $\langle\lambda_i,\alpha_j\rangle=\delta_{ij}$ and $(\lambda_i,\alpha_j)=\delta_{ij}(\alpha_j,\alpha_j)/2$ and that $\delta=\sum\lambda_i$.
For every simple root $\alpha_j$ we have 
\[ 
 (-\alpha_r+a\lambda_r+\delta,\alpha_j)=-(\alpha_r,\alpha_j)+a(\lambda_r,\alpha_j)+(\delta,\alpha_j) = 
  \begin{dcases*} 
  (a-1)/2(\alpha_r,\alpha_r)& if  $r=j$ \\ 
  -(\alpha_r,\alpha_j)+(\lambda_j,\alpha_j)& if $r\neq j$ 
  \end{dcases*} 
\]
Since it is well known that  $(\alpha_r,\alpha_j)\leq 0$ for $\alpha_r,\alpha_j$ distinct simple roots, the case $r\neq j$ is always positive. If $a>1$ also the $r=j$ case is positive, hence we have that the weight is regular of index 0. On the other hand if $a=1$, $(-\alpha_r+a\lambda_r+\delta,\alpha_r)=0$ hence the weight is singular.

In the case of $-\lambda_r+\lambda'+a\lambda_r+\delta$ we have
\[ 
 (-\lambda_r+\lambda'+a\lambda_r+\delta,\alpha_j)=(a\lambda_r+\lambda'',\alpha_j) = 
  \begin{dcases*} 
  a(\lambda_r,\alpha_r)& if  $r=j$ \\ 
  (\lambda'',\alpha_j)& if $r\neq j$ 
  \end{dcases*} 
\] which is always positive and shows that $-\lambda_r+\lambda'+a\lambda_r+\delta$ is regular of index 0.

Lastly in the case of $-2\lambda_r+\lambda'+a\lambda_r+\delta$ we have 
\[ 
 (-2\lambda_r+\lambda'+a\lambda_r+\delta,\alpha_j)=((a-1)\lambda_r+\lambda'',\alpha_j) = 
  \begin{dcases*} 
  (a-1)(\lambda_r,\alpha_r)& if  $r=j$ \\ 
  (\lambda'',\alpha_j)& if $r\neq j$ 
  \end{dcases*} 
\] and the weight is regular of index 0 if $a>1$ while if $a=1$ it is singular because $(-2\lambda_r+\lambda'+a\lambda_r+\delta,\alpha_r)=0$.
\end{proof}
This analysis of the weights combined with the Borel-Weil Theorem gives
\begin{prop}
\label{sommadiretta}
$
h^0(Y,\Omega^1_Y(a))=\sum_i h^0(Y,G^i\Omega^1_Y(a))
$ for all $a\geq1$. 
\end{prop}
\begin{proof}
By the Borel-Weil Theorem \ref{bw} and Lemma \ref{studiopesi} we have that $H^q(Y,G^i\Omega^1_Y(a))=0$ for $q>0$.
Hence in the first page of the spectral sequence of Proposition \ref{spectral}, $E_1^{i,q-i}=0$ if $q>0$ and
the only non-zero pieces are $E_1^{i,-i}$.
 All the differentials of this spectral sequence are also zero, therefore $H^0(Y,G^i\Omega^1_Y(a))=E_1^{i,-i}=E_2^{i,-i}=\dots=E_\infty^{i,-i}$ and the thesis immediately follows.
\end{proof}

\section{Review of the Generalized adjoint theory}
\label{aggiunta}
The theory of adjoint forms was introduced in \cite{CP} and \cite{PZ} and later developed and applied in \cite{Ra}, \cite{PR}, \cite{CNP}, \cite{victor}, \cite{BGN}, \cite{RZ1}, and recently \cite{CRZ}. The basic idea is to study global $1$-forms and $n$-forms of a complex smooth variety of dimension $n$ to get some information on its infinitesimal deformations.
Even if the variety is regular, that is there are no global $1$-forms, this theory can be generalized and applied, see \cite{RZ2}, \cite{RZ3}. This is the generalization that we need in this paper, hence we will recall the basic construction.
Since adjoint forms are also called Massey products, see \cite{PT}, \cite{RZ4}, in this paper we stick mostly with the second denomination to avoid possible confusion with the adjoint representation of a Lie algebra $\lieg$.

Let $X$ and $\xi$ be respectively a smooth compact complex variety of dimension $m$ and a class $\xi\in \text{Ext}^1(\sF,\sL)$ where 
$\sF$ and $\sL$ are two locally free sheaves on $X$ of rank $n$ and $1$ respectively. Then the extension class $\xi$ gives a rank $n+1$ vector bundle $\sE$ on $X$ which fits in an exact sequence:
\begin{equation}
\label{sequenza}
0\to\sL\to \sE\to\sF\to 0.
\end{equation}

By taking the highest wedge product of (\ref{sequenza}), the determinant sheaf $\det\sF:=\bigwedge^n\sF$ fits into the exact sequence:
\begin{equation}
\label{wedge}
0\to\bigwedge^{n-1}\sF\otimes\sL\to \bigwedge^n\sE\to \det\sF\to 0, 
\end{equation} which still corresponds to $\xi$ via the isomorphisms $\text{Ext}^1(\sF,\sL)\cong\text{Ext}^1(\det\sF,\bigwedge^{n-1}\sF\otimes\sL)\cong H^1(X,\sF^\vee\otimes\sL)$.

 A natural problem is to find conditions on the behaviour of the global sections of the involved vector bundles in order to have the splitting of (\ref{sequenza}). From now on, assume that the connecting homomorphism
 $\delta_\xi \colon H^0(X,\sF)\to H^1(X,\sL)$ has a kernel of dimension sufficiently high. More precisely assume that there exists a subspace $W\subset \ker\delta_\xi$ of dimension $n+1$. Choose a basis $\mathcal{B}:=\{\eta_1,\ldots,\eta_{n+1}\}$ of $W$. By definition we can take liftings $s_1,\ldots,s_{n+1}\in H^0(X,\sE)$ of the sections $\eta_1,\ldots,\eta_{n+1}$. If we consider the natural map
\begin{equation*} 
\Lambda^n\colon \bigwedge^{n}H^0(X,\sE)\to H^0(X,\bigwedge^n\sE)
\end{equation*} we can define the sections
\begin{equation}
\label{Omegai}
\Omega_i:=\Lambda^n(s_1\wedge\ldots\wedge\widehat{s_i}\wedge\ldots\wedge s_{n+1})
\end{equation} for $i=1,\ldots,n+1$. Denote by $\omega_i$, for $i=1,\ldots,n+1$, the corresponding sections in $H^0(X,\det\sF)$. By commutativity between evaluation of wedge product and restriction it easily follows that $\omega_i=\lambda^n(\eta_1\wedge\ldots\wedge\widehat{\eta_i}\wedge\ldots\wedge \eta_{n+1})$, where $\lambda^n$ is the natural morphism 
$$\lambda^n\colon \bigwedge^{n}H^0(X,\sF)\to H^0(X,\det\sF).$$
\begin{defn} We denote by $\lambda^nW$ the subspace of $H^0(X,\det\sF)$ generated by $\omega_1,\ldots,\omega_{n+1}$.
If $\lambda^nW$ is nontrivial we call $D_W$ the fixed divisor of the induced sublinear system $|\lambda^n W|\subset \mP(H^0(X,\det\sF))$ and $Z_W$ the base locus of its moving part $|M_W|\subset\mP(H^0(X,\det\mathcal{F}(-D_W)))$.
\end{defn}

\begin{defn} The form $\Omega\in H^0(X,\det\sE)$ corresponding to $s_1\wedge\ldots\wedge s_{n+1}$ via 
\begin{equation}
\Lambda^{n+1}\colon \bigwedge^{n+1}H^0(X,\sE)\to H^0(X,\det\sE)
\end{equation} is called {\it{generalized adjoint form or Massey product associated to W}}.
\end{defn} 
\begin{rmk}
\label{zeri}
This section vanishes on $D_W$ and $Z_W$, that is $\Omega$ in the image of the natural injection $\det\sE(-D_W)\otimes\sI_{Z_W}\to \det\sE$.  
\end{rmk}

The basic idea is that, in a split exact sequence, generalized Massey products as $\Omega$ do not add any information which is not already given by the top forms $\omega_i\in H^{0}(X, \det\sF) $, $i=1,\dots,n+1$. Conversely, what can we say about the exact sequence if generalized Massey product do not add any information?
More precisely, since ${\rm{det}}(\sE)={\rm{det}}(\sF)\otimes\sL$, we will focus on the condition
\begin{equation}
\label{aggiuntazero2}
\Omega \in \Ima(H^0(X,\sL)\otimes \lambda^nW \to H^0(X,\det\sE)).
\end{equation}  

 By \cite[{Theorem [A]} and {Theorem [B]}]{RZ2} we have:
\begin{thm}\label{theoremA}
Let $\Omega\in H^0(X,\det\sE)$ be a generalized Massey product associated to $W$ as above.
If $\Omega \in \Ima(H^0(X,\sL)\otimes \lambda^nW\to H^0(X,\det\sE))$ 
then $\xi\in\ker(H^1(X,\sF^\vee\otimes\sL)\to H^1(X,\sF^\vee\otimes\sL(D_W)))$. Viceversa if $H^0(X,\sL)\cong H^0(X,\sL(D_W))$ and if 
 $\xi\in\ker(H^1(X,\sF^\vee\otimes\sL)\to H^1(X,\sF^\vee\otimes\sL(D_W)))$, then $\Omega \in \Ima(H^0(X,\sL)\otimes \lambda^nW\to H^0(X,\det\sE))$.
\end{thm}
\begin{rmk}\label{zerozero} Note that if $D_W=0$ the above theorem is a criterion for the vanishing of $\xi$ and the splitting of Sequence \ref{sequenza}.
\end{rmk}

To put things in perspective, in this paper the variety $X$ is a smooth member of the linear system $|\sO_Y(d)|$ and from now on we will assume that $Y$ is not a projective space, since this case has already been covered in \cite{RZ2}, and that $\dim Y\geq 3$ because the case of $X$ smooth curve is also well known. We also assume that the degree of $X$ is $d\geq3$. The class $\xi$ that we want to study is an infinitesimal deformation $\xi\in \text{Ext}^1(\Omega^1_X,\sO_X)\cong H^1(X,T_X)$. This infinitesimal deformation gives an exact sequence
\begin{equation}
\label{estensionedivisore}
0\to\sO_X\to \Omega^1_{\sX}|_X\to\Omega^1_X\to0
\end{equation} but we can not study this sequence directly using the theory of Massey products because the sheaf $\Omega^1_X$ does not have in general global sections, whereas we need at least $\dim X+1$.
The idea is of course to twist this sequence by $\sO_X(a)$ with $a$ a sufficiently large integer such that 
$$
h^0(X,\Omega^1_X(a))\geq\dim X+1.
$$ After doing this we obtain the sequence 
\begin{equation}
\label{twist}
0\to\sO_X(a)\to \Omega^1_{\sX}(a)|_X\to\Omega^1_X(a)\to0
\end{equation}
which still associated to $\xi$ via the isomorphism $\text{Ext}^1(\Omega^1_X,\sO_X)\cong \text{Ext}^1(\Omega^1_X(a),\sO_X(a))$ and therefore it has the same effectiveness for the purpose of studying $\xi$.

Once this is done we are in the setting required to apply Theorem \ref{theoremA}, in fact note that every subspace $W\subset H^0(X,\Omega^1_X(a))$ is also automatically in $\ker\delta_\xi$ by the following easy 
\begin{prop}
\label{sollevabile}
The cohomology map $\delta_\xi\colon H^0(X,\Omega^1_X(a))\to H^1(X,\sO_X(a))$ is the null map, that is $H^0(X,\Omega^1_X(a))=\ker\delta_\xi$.
\end{prop}
\begin{proof}
This is easy because $H^1(X,\sO_X(a))$ vanishes.
In fact by the exact sequence defining $X$
$$
0\to \sO_Y(-d)\to \sO_Y\to \sO_X\to 0
$$ twisted by $a$ it is enough to show that 
$$
h^1(Y,\sO_Y(a))=h^2(Y,\sO_Y(a-d))=0.
$$These vanishings are known, see for example \cite[Theorem 3.1.1]{K}.
\end{proof}

In Sections \ref{sez6} we will find an explicit suitable twist $a\in\mZ$ which will allow us to carry on this method.

\section{The pseudo-Jacobi ring}
\label{green}
Recall the following definition from \cite{green1}
\begin{defn}
We say that a property holds for a sufficiently ample line bundle $L$ on a projective variety $X$ if there exists an ample line bundle $L_0$ such that the property holds for all $L$ with $L\otimes L_0^{-1}$ ample.
\end{defn}

Take an $n$-dimensional smooth variety $Y$ and a sufficiently ample line bundle $L$ on $Y$. Let $\tau\in H^0(Y,L)$ be a global section and $X$ the corresponding divisor $(\tau=0)$. Assume that $X$ is smooth.

In this case the usual Jacobian ideal can be replaced by the so called pseudo-Jacobi ideal introduced in \cite{green1} and \cite{green2}. 

Given the sheaf $T_Y$ of regular vector fields, consider the extension
\begin{equation}
0\to \sO_Y\to \Sigma_L\stackrel{\tau}{\rightarrow} T_Y\to 0
\label{prolongation}
\end{equation} with extension class $-c_1(L)\in H^1(Y,\Omega^1_Y)$. $\Sigma_L$ is a sheaf of differential operators of order less or equal to $1$ on the sections of $L$.
In an open subset of $Y$ with coordinates $x_1,\ldots,x_n$ this sheaf is free and is generated by the constant section $1$ and the sections $D_i$, for $i=1,\dots,n$, which operates on the sections of $L$ by
\begin{equation*}
D_i(f\cdot l)=\frac{\partial f}{\partial x_i}\cdot l
\end{equation*}  where $l$ is a trivialization of $L$. The operators $D_i$ are sent to $\frac{\partial}{\partial x_i}$ in $T_Y$.

 In particular to a global section $\tau$ of $L$, we can associate a global section $\widetilde{d\tau}$ of $L\otimes\Sigma_L^\vee$. If locally $\tau=f\cdot l$, then $\widetilde{d\tau}$ is given by
\begin{equation}\label{diffs}
\widetilde{d\tau}=f\cdot l\cdot 1^\vee+\sum_{i=1}^n \frac{\partial f}{\partial x_i}\cdot l\cdot D_i^\vee
\end{equation} where $\{1^\vee,D_1^\vee,\ldots,D_n^\vee\}$ is a local basis of $\Sigma_L^\vee$ dual to $\{1,D_1,\ldots,D_n\}$.

Given a line bundle $E$, the contraction by $\widetilde{d\tau}$ gives a map
\begin{equation*}
E\otimes \Sigma_L\otimes L^\vee\to E.
\end{equation*} 
%To give an idea in the case $E=\sO_X$, the contraction $\Sigma_L\otimes L^\vee\to \sO_Y$ is given explicitly in local coordinates by 
%\begin{equation*}
%a_0\cdot l^\vee\otimes 1+\sum_{i=1}^n a_i \cdot l^\vee\otimes D_i\mapsto a_0\cdot f+\sum_{i=1}^n a_i\cdot\frac{\partial f}{\partial x_i}.
%\end{equation*} 

\begin{defn}
The \emph{pseudo-Jacobi ideal} $\sJ_{E,\tau}$ is the image of the map 
\begin{equation}
H^0(Y,E\otimes \Sigma_L\otimes L^\vee)\to H^0(Y,E).
\end{equation} 
The quotient $H^0(Y,E)/\sJ_{E,\tau}$ is denoted by $R_{E,\tau}$ and is called \emph{Pseudo-Jacobian ring}.
\end{defn}
%and is a piece of a ring graded $\sR:=\oplus_{m\geq 0}R_{\sL^{\otimes m},\sigma}$ called pseudo-Jacobi ring of $E$ with respect to $L$.
The $k$-graded piece of the usual Jacobian ideal of a homogeneous polynomial $F$ of degree $d$ is recovered taking $L=\sO_{\mP^n}(d)$ and $E=\sO_{\mP^n}(k)$. In this case it easy to see that $\Sigma_L=\oplus_{i=1}^{n+1}\sO_{\mP^n}(1)$ and sequence (\ref{prolongation}) is the Euler sequence 
\begin{equation}
0\to \sO_{\mP^n}\to\bigoplus^{n+1}\sO_{\mP^n}(1)\to T_{\mP^n}\to0.
\end{equation} The pseudo-Jacobi ideal $\sJ_{\sO_{\mP^n}(k),F}\subset H^0(\mP^n,\sO_{\mP^n}(k))$ is generated by $\frac{\partial F}{\partial x_0},\ldots,\frac{\partial F}{\partial x_n}$, that is it is the degree $k$ part of the Jacobian ideal.

%In the case of a smooth algebraic variety $Y$ of dimension $n$ with a smooth hypersurface $X$, we take $L$ to be the sheaf $\sO_Y(X)$, the section $\sigma\in H^0(Y,L)$ is such that $X=\text{div}(\sigma)$, and $E=L=\sO_Y(X)$. We consider the deformations of $X$ inside of the ambient space $Y$. Exactly as in the case of projective hypersurfaces, such an infinitesimal deformation of $X$ is given by $X+tR=0$, $t^2=0$, where $R\in H^0(Y,L)$.
%Define $\text{Aut}(Y,L)=\{f\colon Y\to Y\text{ such that } f^*(L)=L\}$. The base of the Kuranishi family for $X$
%is $|L|/\text{Aut}(Y,L)$ and we have
%\begin{prop}\label{tang}
%The tangent space to $|L|/\text{Aut}(Y,L)$ at $X$ is $R_{L,\sigma}$.
%\end{prop} \begin{proof} See \cite[Corollary page 48]{green2}.\end{proof}

So in our case where $Y=G/P$ and $X$ is a smooth member of the linear system $|\sO_Y(d)|$, we will take $L=\sO_Y(d)$ and the section $\tau\in H^0(Y,\sO_Y(d))$ is such that $X=\text{div}(\tau)$.
Since $\pic(Y)=\mZ$, the line bundle $E$ is also identified by an integer $E=\sO_Y(s)$ and for simplicity we will denote $\sJ_{E,\tau}$ by $\sJ_s$ and $R_{E,\tau}$ by $R_s$.
The ring $R=\bigoplus_{s\geq0}R_s$ is simply called Jacobian ring of $X$ and $\sJ=\bigoplus_{s\geq0}\sJ_s$ the associated Jacobian ideal.

Among these rings, the one coming from the choice $E=L=\sO_Y(d)$ is of particular interest. Consider the deformations of $X$ inside of the ambient space $Y$. Exactly as in the case of projective hypersurfaces, such an infinitesimal deformation of $X$ is given by $X+tR=0$, $t^2=0$, where $R\in H^0(Y,\sO_Y(d))$.
\begin{prop}
\label{etutto}
Assume that $Y=(\lieg,\alpha_r)$ and $X$ is a smooth hypersurface of degree $d\geq3$.
We have the identifications 
\begin{equation}
\label{identita}
R_d\cong H^1(X,T_X)\cong H^1(Y,T_Y(-\textnormal{log}\ X))\cong H^0(X,N_{X/Y})/\Ima H^0(Y,T_{Y}).
\end{equation}
\end{prop}
\begin{proof}
We recall that $T_Y(-\textnormal{log}\ X)$ is the subsheaf of the tangent sheaf $T_Y$ consisting of those derivations that preserve the ideal sheaf $\sO_Y(-X)$ and it fits into the following commutative diagram
\begin{equation}
\xymatrix{
&&0\ar[d]&0\ar[d]&\\
0\ar[r]&T_Y(-X)\ar[r]\ar@{=}[d]&T_Y(-\text{log}\ X)\ar[d]\ar[r]&T_X\ar[r]\ar[d]&0\\
0\ar[r]&T_Y(-X)\ar[r]&T_Y\ar[d]\ar[r]&T_{Y}|_X\ar[r]\ar[d]&0\\
&&N_{X/Y}\ar@{=}[r]\ar[d]&N_{X/Y}\ar[d]&\\
&&0&0&
} \end{equation}
By \cite[Theorem 3.1.1]{K}, we have the vanishings $H^1(T_Y)=H^1(T_Y(-X))=H^2(T_Y(-X))=0$ which give isomorphisms 
$H^1(Y,T_Y(-\textnormal{log}\ X))\cong H^0(X,N_{X/Y})/\Ima H^0(Y,T_{Y})$ and $H^1(X,T_X)\cong H^1(Y,T_Y(-\textnormal{log}\ X))$.

The last isomorphism $R_d\cong H^1(Y,T_Y(-\textnormal{log}\ X))$ comes from the diagram 
\begin{equation}
\xymatrix{
&&0&0&\\
0\ar[r]&T_Y(-\text{log}\ X)\ar[r]\ar@{=}[d]&T_Y\ar[u]\ar[r]&N_{X/Y}\ar[r]\ar[u]&0\\
0\ar[r]&T_Y(-\text{log}\ X)\ar[r]&\Sigma\ar[u]\ar^-{\widetilde{d\tau}}[r]&\sO_Y(d)\ar[r]\ar[u]&0\\
&&\sO_Y\ar@{=}[r]\ar[u]&\sO_Y\ar[u]&\\
&&0\ar[u]&0\ar[u]&
} \end{equation}
By the vanishing $H^1(T_Y)=H^1(\sO_Y)=0$ (see again \cite[Theorem 3.1.1]{K}), we get $H^1(\Sigma)=0$ and therefore the cohomology sequence is exact
$$
H^0(\Sigma)\stackrel{\widetilde{d\tau}}{\rightarrow} H^0(\sO_Y(d))\to H^1(T_Y(-\text{log}\ X))\to 0.
$$
\end{proof}
\begin{rmk}
Since $H^1(\sO_Y)=0$ we have $H^0(N_{X/Y})\cong H^0(\sO_Y(d))/\mC$, that is all the infinitesimal deformations of $X$ are inside $Y$.

It is also standard in this setting (see for example \cite[Proposition 9.1.9]{K}) that if $U_d$ denote the Zariski open subset parametrizing all smooth hypersurfaces of degree $d$, then there exist an open $\text{Aut}(Y)$-invariant subset $U^0_d$ such that $U^0_d/\text{Aut}(Y)$ is smooth. If $[X]$ is a class of a smooth hypersurface, then the tangent space at $[X]$ is $R_d$.  
\end{rmk}

\section{The dimension of \texorpdfstring{$H^0(X,\Omega^1_X(2))$}{TEXT} }
\label{sez6}
As discussed in Section \ref{aggiunta}, we want to use the theory of generalized Massey product, but to do so we have to choose an appropriate $a\in\mZ$ such that $H^0(X,\Omega^1_X(a))$ has a high enough number of liftable sections.
As pointed out in Proposition \ref{sollevabile}, the \lq\lq liftable\rq\rq\ part is not a problem, therefore we only need to find $a$ such that $H^0(X,\Omega^1_X(a))\geq\dim X+1$.

Both in the case of smooth hypersurfaces of the projective spaces and of Grassmannians, studied in \cite{RZ2} and in \cite{RZ3} respectively, the choice $a=2$ does the trick. Hence $a=2$ seems the natural starting point to look at.

We will start with the following
\begin{lem}
\label{isocom}
If $X$ is a smooth hypersurface of degree $d\geq3$ in $Y$, $\dim Y\geq3$, then $H^0(Y,\Omega^1_Y(2))=H^0(X,\Omega^1_X(2))$.
\end{lem}
\begin{proof}
We look at the short exact sequences 
\begin{equation}
0\to \sO_X(2-d)\to \Omega^1_{Y}(2)|_X\to \Omega^1_X(2)\to 0
\end{equation} and 
\begin{equation}
0\to \Omega^1_Y(2-d)\to \Omega^1_{Y}(2)\to \Omega^1_{Y}(2)|_X\to 0
\end{equation}
Using the first we get the isomorphism $H^0(X,\Omega^1_{Y}(2)|_X)=H^0(X,\Omega^1_X(2))$ since $H^i(\sO_X(2-d))=0$ for $0\leq i\leq \dim X-1$ by the sequence $0\to \sO_Y(2-2d)\to \sO_Y(2-d)\to \sO_X(2-d)\to 0$ and the vanishings of \cite[Theorem 3.1.1]{K}.
In the same fashion from the second we get the isomorphism $H^0(Y,\Omega^1_{Y}(2))=H^0(X,\Omega^1_{Y}(2)|_X)$ since $H^i(Y,\Omega^1_Y(2-d))=0$ for $0\leq i\leq \dim X-1$.
\end{proof}
By this simple Lemma it will be enough to look at the dimension of the spaces $H^0(Y,\Omega^1_{Y}(2))$ and this will be done using the machinery introduced in Section \ref{cotangent} and in particular in Proposition \ref{sommadiretta}.
\begin{thm}
\label{lungo}
Let $Y$ given by $(\lieg,\alpha_r)$ be one of the varieties in Table \ref{tabellapesi}, $\dim Y\geq3$, and let $X$ be a smooth hypersurface of degree $d\geq3$ in $Y$ as above, then $h^0(Y,\Omega^1_Y(2))=h^0(X,\Omega^1_X(2))\geq\dim X+1=\dim Y$.
\end{thm}
\begin{proof}
Since the computations are often similar, we only show some meaningful cases and simply give an overview of the others. We also show how in some cases we can use classically know cohomological dimensions as $h^0(Y,\sO_Y(1))$ and $h^0(Y,\Omega_Y^1(1))$.

\emph{Case $(A_l,\alpha_r)$ with $1\leq r\leq l$}
\newline
These are Grassmannian varieties. The result for this class is the content of \cite{RZ3}, see also \cite[Theorem page 169]{snow}.

\emph{Case $(B_l,\alpha_1)$}
\newline
These are quadric hypersurfaces $Q=Q^N$ in $\mP^{N+1}$, $N=2l-1$. The result can be easily obtained by a cohomological analysis of the twisted conormal exact sequence $$0\to \sO_Q\to \Omega^1_{\mP^{N+1}}(2)|_Q\to \Omega^1_Q(2)\to 0.$$
Using Bott vanishing theorem (see \cite{bott}) we get $h^0(\Omega^1_Q(2))=\binom{N+2}{2}$ and we immediately see that this is greater than $\dim Q=N$.

\emph{Case $(B_l,\alpha_r)$ with $r=2,3$}
\newline
For this case the computation is very similar to the case $(C_l,\alpha_r)$ that we will explicitly see later. Hence we will omit it here.

\emph{Case $(B_l,\alpha_r)$ with $4\leq r\leq l-1$}
\newline
Also in this case the dimension can be directly computed, but since Konno in \cite[Table 3]{K1} computes the dimension $h^0(Y,\Omega^1_Y(1))=\binom{2l+1}{r-2}$ and since $h^0(Y,\Omega^1_Y(2))\geq h^0(Y,\Omega^1_Y(1))$ it is enough to show that $h^0(Y,\Omega^1_Y(1))\geq \dim Y$. We have that $\dim Y=2r(l-r)+r(r+1)/2$ hence the inequality to prove is
$$
\binom{2l+1}{r-2}\geq 2r(l-r)+r(r+1)/2.
$$ For simplicity we prove that 
$$
\binom{2l+1}{r-2}\geq 2r(l-r)+r(r+1)=r(2l-r+1)
$$ which of course implies the previous inequality. Developing the binomial coefficient we have 
$$
\frac{(2l+1)(2l)\cdots(2l-r+4)}{(r-2)!}\geq r(2l-r+1)
$$ and since $2l-r+4>2l-r+1$ this can be simplified to 
$$
\frac{(2l+1)(2l)\cdots(2l-r+5)}{(r-2)!}\geq r.
$$ Dividing by $r$ we get
$$
\frac{2l+1}{r}\cdot\frac{2l}{r-2}\cdots\frac{2l-r+5}{3}\cdot\frac{1}{2}\geq 1
$$
By $r\leq l-1$ all the fractions except $1/2$ are strictly greater than 2, completing this case.

\emph{Case $(C_l,\alpha_r)$ with $2\leq r\leq l-1$}
\newline
Here $h^0(Y,\Omega^1_Y(1))=0$ hence we have to directly tackle the sheaf $\Omega^1_Y(2)$.

Let's briefly recall the root system of  $(C_l,\alpha_r)$, see \cite[Section 12]{Hump}. Denoting by $\e_j$ the standard basis of $E=\mR^l$, the vectors $\pm2\e_i$ and $\pm(\e_i\pm\e_j)$, $i\neq j$, form a root system of $C_l$ with basis $\{\e_1-\e_2,\dots,\e_{l-1}-\e_l,2\e_l\}$. Hence we will denote $\alpha_i=\e_i-\e_{i+1}$, $i=1,\dots,l-1$, and $\alpha_l=2\e_l$.

For convenience we will divide the positive roots into three subsets:
\begin{itemize}
	\item[-] $\sum_{k=i}^{j-1}\alpha_k$ with $i=1,\dots,l-1$ and $j>i$ which correspond to $\e_i-\e_j$.
	\item[-] $\alpha_l$ and $2\sum_{k=i}^{l-1}\alpha_k+\alpha_l$ with $i=1,\dots,l-1$ which correspond to $2\e_i$.  
	\item[-] $\sum_{k=i}^{j-1}\alpha_k+2\sum_{k=j}^{l-1}\alpha_k+\alpha_l$ with $i=1,\dots,l-1$, $j>i$ which correspond to $\e_i+\e_j$.
\end{itemize} We denote these sets by $A$, $B$ and $C$ respectively and we have that $A\cup B\cup C=\Phi^+$. Note that the roots of $A$ and $C$ are short while the roots of $B$ are long.
It is not difficult to see that we have 
\[ 
 \langle \delta, \alpha\rangle = 
  \begin{dcases*} 
  j-i& if  $\alpha\in A$ \\ 
  l-i+1& if $\alpha\in B$\\
	2l-i-j+2& if $\alpha \in C$
  \end{dcases*} 
\]

By Table \ref{tabellapesi} and Proposition \ref{sommadiretta}, $h^0(\Omega^1_Y(2))$ only has two summands, the first corresponding to $G^1\Omega^1_Y(2)$ with lowest weight $\alpha_r-2\lambda_r$ and the second corresponding to $G^2\Omega^1_Y(2)$ with lowest weight $2\lambda_r-2\lambda_{r-1}-2\lambda_r=-2\lambda_{r-1}$. Hence it is enough to show that $h^0(G^i\Omega^1_Y(2))\geq\dim Y$ for $i=1$ or $i=2$. In this case we will prove that $h^0(G^2\Omega^1_Y(2))\geq\dim Y$.

We compute this dimension using the Weyl formula (\ref{weilformula}) and, as a further simplification, we will run the product of this formula not on all $\alpha\in \Phi^+$, but just on $\alpha\in A$, that is 
\begin{equation}
\prod_{\alpha\in A}\frac{\langle 2\lambda_{r-1}+\delta,\alpha \rangle}{\langle \delta,\alpha \rangle}=\prod_{\alpha\in A}\frac{2\langle \lambda_{r-1},\alpha \rangle+\langle \delta,\alpha \rangle}{\langle \delta,\alpha \rangle}
\label{cl}
\end{equation}
Now if $\alpha=\sum_{i=1}^{j-1}\alpha_i$ we have that $\langle \delta,\alpha \rangle=j-i$ and 
\[ 
  \langle \lambda_{r-1},\alpha \rangle= 
  \begin{dcases*} 
  1 & if  $\alpha_{r-1}$ appears in $\alpha$ \\ 
  0& otherwise
  \end{dcases*} 
\] hence the only $\alpha\in A$ that give a contribution in (\ref{cl}) are those with $i\leq r-1\leq j-1$. For these roots we have $\frac{2\langle \lambda_{r-1},\alpha \rangle+\langle \delta,\alpha \rangle}{\langle \delta,\alpha \rangle}=\frac{2+j-i}{j-i}$ and (\ref{cl}) becomes 
\begin{equation}
\begin{split}
\prod_{\alpha\in A}\frac{2\langle \lambda_{r-1},\alpha \rangle+\langle \delta,\alpha \rangle}{\langle \delta,\alpha \rangle}=\prod_{i=1}^{r-1}\prod_{j=r}^l \frac{2+j-i}{j-i}=\prod_{i=1}^{r-1}\frac{(2+r-i)(3+r-i)\cdots(2+l-i)}{(r-i)(r+1-i)\cdots(l-i)}=\\
=\prod_{i=1}^{r-1}\frac{(1+l-i)(2+l-i)}{(r-i)(r+1-i)}=\frac{(l+1)(l)\cdots(l-r+3)}{r!}\cdot\frac{(l)(l-1)\cdots(l-r+2)}{(r-1)!}
\end{split}
\label{cl2}
\end{equation}
It remains to prove that this is bigger than the dimension of $Y$ which is $\dim Y=2r(l-r)+r(r+1)/2$. It is in fact greater than $2r(l-r)+r(r+1)=r(2l-r+1)$, that is 
$$
\frac{(l+1)(l)\cdots(l-r+3)}{r!}\cdot\frac{(l)(l-1)\cdots(l-r+2)}{(r-1)!}\geq r(2l-r+1).
$$ This is equivalent to
$$
\frac{(l+1)(l)\cdots(l-r+3)}{r!}\cdot\frac{(l)(l-1)\cdots(l-r+2)}{r!}\geq (2l-r+1)
$$ and, recalling that $r\leq l-1$, it is not difficult to see that this inequality holds for all $l\geq3$, maybe with some attention in the case of low $l$. 

\emph{Case $(C_l,\alpha_l)$}
\newline
This is the last case of the $C_l$ class. We point out that in this case we are dealing with a Hermitian symmetric space and $\Omega^1_Y(2)$ is irreducible with lowest weight $\alpha_{l}-2\lambda_l$.

Hence $h^0(Y,\Omega^1_Y(2))$ is given by 
\begin{equation}
\label{clrl}
\prod_{\alpha\in \Phi^+}\frac{\langle-\alpha_l+2\lambda_l+\delta,\alpha\rangle}{\langle\delta,\alpha\rangle}.
\end{equation} We can work as in the previous case and restrict the product on the roots of $A$. Hence if $\alpha\in A$ is $\alpha=\sum_{k=i}^{j-1}\alpha_k$, $i=1,\dots,l-1$ and $j>i$, we have that $\langle\delta,\alpha\rangle=j-i$ as before, $\langle\lambda_l,\alpha\rangle=0$ and 
\[ 
  \langle \alpha_l,\alpha \rangle= 
  \begin{dcases*} 
  -2& if  $\alpha_{l-1}$ appears in $\alpha$ i.e. $j=l$ \\ 
  0& otherwise
  \end{dcases*} 
\] Therefore the only roots giving a contributions are $\alpha=\sum_{k=i}^{l-1}\alpha_k$ and in this case $\frac{\langle-\alpha_l+2\lambda_l+\delta,\alpha\rangle}{\langle\delta,\alpha\rangle}=\frac{2+l-i}{l-1}$ and (\ref{clrl}) is 
$$
\prod_{i=1}^{l-1}\frac{2+l-i}{l-1}=\frac{(l+1)l\cdots3}{(l-1)(l-2)\cdots1}=\frac{l(l+1)}{2}.
$$ The dimension of our space in this case is also exactly $\frac{l(l+1)}{2}$ hence we are done since proving the equality is enough. Nevertheless note that if we add also the roots of $B$ and $C$ in the computation we will get a strict inequality $h^0(Y,\Omega^1_Y(2))>\dim Y$.

\emph{Case $(D_l,\alpha_1)$}
\newline
This is a quadric hypersurface $Q^N$ in $\mP^{N+1}$ with $N=2l-2$. The computations are the same as the case $(B_l,\alpha_1)$.

\emph{Case $(D_l,\alpha_r)$ with $4\leq r\leq l-1$}
\newline
Of course it is possible to use the previous general strategy to compute $h^0(Y,\Omega_Y^1(2))$ but as in the case of $(B_l,\alpha_r)$, $4\leq r\leq l-1$, it is also possible to note that $h^0(Y,\Omega^1_Y(1))$ is already greater than the dimension.
In fact here $h^0(Y,\Omega^1_Y(1))=\binom{2l}{r-2}$ and $\dim Y=2r(l-r)+r(r-1)/2$ (see \cite[Table 3]{K1}) and the inequality is very similar to the one seen before.

\emph{Case $(D_l,\alpha_r)$ with $r=3,4$ and $r=l$}
\newline
Here we estimate $h^0(Y,\Omega^1_Y(2))$ directly and as before it is convenient to partition the positive roots into subsets of $\Phi^+$.
Using the same notation as before, the root system of  $(D_l,\alpha_r)$ is given by the vectors $\pm(\e_i\pm\e_j)$, $i\neq j$, see \cite[Section 12]{Hump}. For basis we take $\alpha_i=\e_i-\e_{i+1}$, $i=1,\dots,l-1$, and $\alpha_l=\e_{l-1}+\e_l$. 

This time we divide the positive roots into two subsets:
\begin{itemize}
	\item[-] $\sum_{k=i}^{j-1}\alpha_k$ with $i=1,\dots,l-1$ and $j>i$ which correspond to $\e_i-\e_j$.
	\item[-] $\alpha_l$, $\sum_{k=i}^{l-2}\alpha_k+\alpha_l$ and $\sum_{k=i}^{l-2}\alpha_k+\alpha_l+\sum_{k=j}^{l-1}\alpha_k$ with $i=1,\dots,l-2$, $j>i$ which correspond to $\e_i+\e_j$.
\end{itemize} We denote these sets by $A$, $B$. Note that the all the roots have the same length and 
\[ 
  \langle \delta,\alpha \rangle= 
  \begin{dcases*} 
  j-i & if  $\alpha\in A$ \\ 
  2l-i-j& if $\alpha\in B$.
  \end{dcases*} 
\] 
The use of the Weyl formula and the following inequality are very similar to previous cases and they immediately give the result.

\emph{Exceptional cases $E_6, E_7, E_8, F_4$, and $G_2$}
\newline 
Here the situation is a little easier because the dimensions are explicit and not functions of $r$ and $l$. We can mostly divide all these cases into three subgroups.

\emph{First Group: $(E_6,\alpha_r)$ with $r=3,4$, $(E_7,\alpha_r)$ with $r=2,\ldots,6$, $(E_8,\alpha_r)$ with $r=1,\dots,7$, $(F_4,\alpha_r)$ with $r=2,3$.}
\newline 
In all these cases $h^0(Y,\Omega^1_Y(1))>\dim Y$ and as we have seen before this is enough. For these dimension we again refer to \cite[Table 3]{K1}.

\emph{Second Group: $(E_6,\alpha_2)$, $(E_7,\alpha_1)$, $(E_8,\alpha_8)$, $(F_4,\alpha_1)$, $(G_2,\alpha_2)$.}
\newline 
By Table \ref{tabellapesi}, note that in all these cases $(\lieg,\alpha_r)$ we have that $\lambda_r$ is the lowest weight of one of the $G^i\Omega^1$, hence $\lambda_r-2\lambda_r=-\lambda_r$ is the lowest weight of the corresponding $G^i\Omega^1(2)$. Since $-\lambda_r$ is also the lowest weight of $\sO_Y(1)$ it follows that $h^0(Y,G^i\Omega^1(2))=h^0(Y,\sO(1))$. These dimensions are all known, see for example \cite[Table 1]{K1}, and they are all greater than $\dim Y$.

\emph{Third Group: $(E_6,\alpha_1)$, $(E_7,\alpha_7)$, $(E_8,\alpha_8)$, $(F_4,\alpha_4)$.}
\newline
Here we can only explicitly compute the dimension. The roots are well known, for example see \cite[Section 12]{Hump} or \cite[Appendix]{K} for a list.
We report the case $(F_4,\alpha_4)$, the other are very similar. The elements of $\Phi$ in $\mR^4$ are given by the vectors $\pm\e_i$, $\pm(\e_i\pm\e_j)$, $i\neq j$, $\pm\frac{1}{2}(\e_1\pm\e_2\pm\e_3\pm\e_4)$. We take $\{\e_2-\e_3,\e_3-\e_4,\e_4,\frac{1}{2}(\e_1-\e_2-\e_3-\e_4)\}$ as a basis of simple roots. 
By Table \ref{tabellapesi} we see that $G^1\Omega^1_Y(2)$ has lowest weight $\alpha_4-2\lambda_4$ and $G^2\Omega^1_Y(2)$ has lowest weight $-\lambda_1$. Hence by Weyl formula 
$$
h^0(Y,G^1\Omega^1_Y(2))=\prod_{\alpha\in \Phi^+}\frac{\langle-\alpha_4+2\lambda_4+\delta,\alpha\rangle}{\langle\delta,\alpha\rangle}=273
$$
and
$$
h^0(Y,G^2\Omega^1_Y(2))=\prod_{\alpha\in \Phi^+}\frac{\langle\lambda_1+\delta,\alpha\rangle}{\langle\delta,\alpha\rangle}=52
$$ hence $h^0(Y,\Omega^1_Y(2))=325>\dim Y=15$.
\end{proof}

\section{Infinitesimal Torelli and Massey Product}
\label{sez7}
In this section we finally put together the theory of infinitesimal Torelli theorem and Massey products.
As anticipated in the previous sections,  we take $X$ an $(N-1)$-dimensional variety which is a smooth member of the linear system $|\sO_Y(d)|$, $d\geq3$, and a class $\xi$ giving an infinitesimal deformation $\xi\in \text{Ext}^1(\Omega^1_X,\sO_X)\cong H^1(X,T_X)$
\begin{equation}
\label{estensionedivisore1}
0\to\sO_X\to \Omega^1_{\sX}|_X\to\Omega^1_X\to0.
\end{equation}
By Theorem \ref{lungo} and Proposition \ref{sollevabile} we now that if we twist this sequence by $\sO_X(2)$
\begin{equation}
\label{twist1}
0\to\sO_X(2)\to \Omega^1_{\sX}(2)|_X\to\Omega^1_X(2)\to0
\end{equation}
we have that $h^0(X,\Omega^1_X(2))\geq\dim X+1=N$ and all the global sections of $H^0(X,\Omega^1_X(2))$ are liftable to $H^0(X,\Omega^1_{\sX}(2)|_X)$.
The key point is considering, together with (\ref{estensionedivisore1}), the conormal exact sequence 
\begin{equation}
0\to\sO_X(-d)\to \Omega^1_{Y}|_X\to\Omega^1_X\to0
\end{equation}
also twisted by $\sO_X(2)$. These sequences give the diagram 
\begin{equation}
\label{diagramma6}
\xymatrix{
&&0&0&\\
0\ar[r]&\sO_X(2)\ar[r]&\Omega^1_{\sX}(2)|_X\ar[u]\ar[r]&\Omega^1_X(2)\ar[u]\ar[r]&0\\
0\ar[r]&\sO_X(2)\ar[r]\ar@{=}[u]&\sG\ar[r]\ar[u]&\Omega^1_{Y}(2)|_X\ar[u]\ar[r]&0\\
&&\sO_X(2-a)\ar[u]\ar@{=}[r]&\sO_X(2-d)\ar[u]&\\
&&0\ar[u]&0.\ar[u]&}
\end{equation}

By Proposition \ref{etutto} our deformation $\xi$ comes from an element $R\in H^0(Y,\sO_Y(d))$, hence the middle horizontal sequence completing diagram (\ref{diagramma6}) is associated to the zero element of $H^1(X,T_{Y}|_X)$. Therefore we have the splitting of the second row and we can find a diagonal splitting 
$$\phi\colon \Omega^1_{Y}(2)|_X\to \Omega^1_{\sX}(2)|_X.$$
Since $\det(\Omega^1_{Y}(2)|_X)\cong \Omega_X^{N-1}(2N-d)$ and $\det(\Omega^1_{\sX}(2)|_X)\cong \Omega^{N-1}_X(2N)$, it is not difficult to see that $\phi$ at the level of global top forms induces a map 
\begin{equation*}
\phi^n\colon H^0(X,\Omega_X^{N-1}(2N-d))\to H^0(X,\Omega_X^{N-1}(2N))
\end{equation*} given by multiplication with the section $R|_{X}\in H^0(X,\sO_X(d))$, see \cite[Proposition 5.2.1]{RZ3}.

Now since $H^1(\sO_X(2-d))$ vanishes, all the global sections in $H^0(X,\Omega_X^1(2))$ can be lifted not only to $H^0(X,\Omega^1_{\sX}(2)|_X)$ but also to $H^0(Y,\Omega^1_{\sX}(2)|_X)$.

Consider $\dim X+1=N$ global sections $\eta_1,\ldots,\eta_N\in H^0(X,\Omega^1_X(2))$ which have unique liftings $\tilde{s}_1,\ldots,\tilde{s}_N\in H^0(X,\Omega^1_{Y}(2)|_X)$. Call $\widetilde{\Omega}\in H^0(X,\Omega_X^{N-1}(2N-d))$ the generalized Massey product corresponding to $\tilde{s}_1\wedge\dots\wedge \tilde{s}_N$. If we take $s_1:=\phi(\tilde{s}_1),\ldots,s_N:=\phi(\tilde{s}_N)\in H^0(X,\Omega^1_{\sX}(2)|_X)$, we have that the generalized Massey product $\Omega\in H^0(X,\det (\Omega^1_{\sX}(2)|_X))=H^0(X,\Omega_X^{N-1}(2N))$ corresponding to $s_1\wedge\dots\wedge s_N$ is $\Omega=\widetilde{\Omega}\cdot R$. We point out that $\widetilde{\Omega}$ does not depend on the deformation $\xi$, while $\Omega$ obviously does.

\begin{thm}
\label{teoxi}
Assume that $W=\left\langle \eta_1\ldots,\eta_N\right\rangle$ is a generic subspace in $H^0(X,\Omega^1_X(2))$ with $\lambda^NW\neq0$.
Then $R$ is in the pseudo-Jacobi ideal $\sJ_{d}$ if and only if the Massey product $\Omega$ is in the image of $H^0(X,\sO_X(2))\otimes \lambda^NW\to H^0(X,\Omega_X^{N-1}(2N))$.
\end{thm}
\begin{proof}
First we recall that $\Omega^1_X(2)$ is generated by its global sections, see for example \cite[Page 165]{snow}.
Hence it follows by \cite[Proposition 3.1.6]{PZ} that the fixed divisor $D_W$ is zero for generic $W$. 
Under this condition, Theorem \ref{theoremA} gives an equivalence; see also Remark \ref{zerozero}. Thus if $$\Omega\in \Ima H^0(X,\sO_X(2))\otimes \lambda^NW\to H^0(X,\Omega_X^{N-1}(2N))$$ then 
$\xi\in\ker(H^1(X,T_{X})\to H^1(X,T_{X}(D_W)))$, that is $\xi=0$ since $D_W=0$. Hence $R$ is zero in $R_d$, that is $R\in \sJ_d$.

Conversely if $R$ is in the pseudo-Jacobi ideal, then the deformation $\xi$ is zero. In particular $\xi\in\ker(H^1(X,T_{X})\to H^1(X,T_{X}(D_W)))$ and since $D_W=0$ 
 by the converse in Theorem \ref{theoremA} it follows $\Omega\in \Ima H^0(X,\sO_X(2))\otimes \lambda^NW\to H^0(X,\Omega_X^{N-1}(2N))$.
\end{proof}

We recall the following two results of Konno \cite{K}.  
The first is a version of Macaulay's Theorem, see \cite[Theorem 6.3.2]{K}.
Denote by $\sigma$ the integer $\sigma=(N+1)d-2k(Y)$ where $k(Y)$ is defined by $K_Y=\sO_Y(-k(Y))$.

\begin{thm}
\label{macaulay}
Let $Y$ one of the spaces in Table \ref{tabellapesi} and let $R=\bigoplus_{a\geq0} R_a$ be the Jacobian ring of a smooth hypersurface of degree $d\geq3$. 
Then, $R_\sigma\cong\mC$.
Furthermore for each integer $a$ such that $0 \leq a \leq d$, consider the natural duality
pairing
\begin{equation}
R_a\otimes R_{\sigma-a}\to R_{\sigma}\cong\mC
\label{pairing}
\end{equation}
This pairing induces an exact sequence $0\to R_{\sigma-a}\to R_a^\vee$ except the following cases 
\begin{itemize}
	\item[-] $Y=Q^N$: $(d, a)=(d, d-2), (3, 2), (3, 3), (4, 4)$,
	\item[-] $Y=(A_5,\alpha_2),(E_6, \alpha_1)$: $(d, a)=(3, 3)$,
	\item[-] $Y=(C_3, \alpha_2),(F_4, \alpha_4)$: $(d, a)=(d, d-1), (3, 3)$,
	\item[-] $Y=(C_l, \alpha_2); l\geq4$: $(d, a)=(d, d-1)$.
\end{itemize}
Furthermore, it induces an exact sequence $R_{\sigma-a}\to R_a^\vee\to0$
except when $Y=Q^3$ and $(d, a)=(3, 3)$.
\end{thm}

The second result is the infinitesimal Torelli theorem, see \cite[Theorem 9.2.6]{K}.

\begin{thm}[Infinitesimal Torelli theorem]
\label{torelli}
Let $Y$ one of the spaces in Table \ref{tabellapesi} and $X$ a smooth hypersurface of degree $d\geq3$. Then the infinitesimal period map is injective.
\end{thm}

Twisted forms provide a natural setting to apply the Generalized Adjoint Theory. Indeed $H^0(Y,\det(\Omega^{1}_{Y}(m))) $ is an irreducible representation and since, for every $m\geq 2$, $\Omega^{1}_Y(m)$ is globally generated, we have easily that any element of $H^0 (Y, \det( \Omega^{1}_{Y}(m)))$ is actually obtainable as a $\mathbb C$-linear combination of totally decomposable forms. More precisely we have, in the case $m=2$:
\begin{prop}
\label{surgettivita}
The natural map  induced by the wedge product 
\begin{equation}
\bigwedge^{N}H^0(X,\Omega^1_{Y}(2)|_X)\to H^0(X, \det (\Omega^1_{Y}(2)|_X))
\end{equation}is surjective.
\end{prop}
\begin{proof}
If $Y$ is given by $(\lieg,\alpha_r)$, we know that $\det (\Omega^1_{Y}(2))\cong\sO_Y(-k(Y)+2N)$ is irreducible with lowest weight $(k(Y)-2N)\lambda_r$.
By Borel-Weil theorem, $H^0(Y, \det (\Omega^1_{Y}(2)))$ is also an irreducible representation. 

We have already noticed that $\Omega^1_Y(2)$ is globally generated by its global sections, hence the natural homomorphism
\begin{equation}
\bigwedge^{N}H^0(Y,\Omega^1_{Y}(2))\to H^0(Y, \det (\Omega^1_{Y}(2)))
\end{equation} is non zero and hence it is surjective by Schur's Lemma.

 By Lemma \ref{isocom} we know that $H^0(Y,\Omega^1_{Y}(2))\cong H^0(X,\Omega^1_{Y}(2)|_X)$. Hence the claim follows if we show that $H^0(Y, \det (\Omega^1_{Y}(2)))\to H^0(X, \det (\Omega^1_{Y}(2)|_X))$ is surjective. Indeed this follows by the exact sequence
\begin{equation}
0\to \Omega^N_Y(2N-d)\to \Omega^N_Y(2N)\to \Omega^N_Y(2N)|_X\to0
\end{equation} and the vanishing of $H^1(Y,\Omega^N_Y(2N-d))$ (cf. \cite[Theorem 3.1.1]{K}).
\end{proof}

%\subsection{Volume forms and the infinitesimal Torelli theorem}
We link the global forms of $\Omega_Y^N(2N)$, which are objects coming from the ambient variety $Y$, to the infinitesimal deformations of $X\subset Y$ contained in pseudo-Jacobi ideal $\sJ_{d}$.
\begin{lem}\label{quattro}
Let $Y$ one of the spaces in Table \ref{tabellapesi} and $X=(\tau=0)$ a smooth hypersurface of degree $d\geq3$, except the cases $Y=Q^3$, $d=3$ and $Y=(D_3,\alpha_2)$, $d=3$. Then the infinitesimal deformation $R$ is in the pseudo-Jacobi ideal $\sJ_{d}$ if and only if $R\widetilde{\Omega}\in \sJ_{2N-k(Y)+d}$ for every section $\widetilde{\Omega}\in H^0(Y,\Omega_Y^N(2N))$ which restricts to a generalized Massey product relative to the vertical exact sequence of diagram (\ref{diagramma6})
$$
0\to \sO_X(2-d)\to \Omega^1_{Y}(2)|_X\to \Omega^1_X(2)\to 0
$$
\end{lem}
\begin{proof}
Note that a generalized Massey product of this sequence is in fact an element of $$H^0(X,(\Omega_Y^N(2N))|_X)=H^0(X,\Omega_X^{N-1}(2N-d)).$$

We want to apply the Macaulay's theorem \ref{macaulay}. We only know that $R\widetilde{\Omega}\in \sJ_{2N-k(Y)+d}$ for $\widetilde{\Omega}\in H^0(Y,\Omega_Y^N(2N))$ which restricts to a generalized adjoint but this is in fact enough to have 
\begin{equation*}
R\cdot H^0(Y,\Omega_Y^N(2N))\subset \sJ_{2N-k(Y)+d}.
\end{equation*}
In fact consider the restriction sequence
\begin{equation*}
0\to \Omega_Y^N(2N-d)\to\Omega_Y^N(2N)\to\Omega_Y^N(2N)|_X\to 0.
\end{equation*} 
As we have seen in the proof of Proposition \ref{surgettivita},  
\begin{equation*}
0\to H^0(Y,\Omega_Y^N(2N-d))\to H^0(Y,\Omega_Y^N(2N))\to H^0(X,\Omega_Y^N(2N)|_X)\to0
\end{equation*} is exact. 
By the same proposition we can also assume that all the sections of $H^0(X,\Omega_Y^N(2N)|_X)$ are in fact linear combinations of generalized Massey products. Hence our hypothesis that $R\widetilde{\Omega}\in \sJ_{2N-k(Y)+d}$ for every section $\widetilde{\Omega}\in H^0(Y,\Omega_Y^N(2N))$ as above, together with the fact that the map $$H^0(Y,\Omega_Y^N(2N-d))\to H^0(Y,\Omega_Y^N(2N)))$$ is given by the multiplication by $\tau$, which is an element of the pseudo-Jacobi ideal, implies that 
\begin{equation*}
R\cdot H^0(Y,\Omega_Y^N(2N))\subset \sJ_{2N-k(Y)+d}.
\end{equation*}

Now we apply Macaulay's theorem \ref{macaulay} to deduce that $R$ is in the pseudo-Jacobi ideal. It is enough to show that
\begin{equation}
\label{surjok}
R_{dN-2N-k(Y)}\otimes R_{2N-k(Y)}\to R_{dN-2k(Y)}
\end{equation} is surjective. In fact if this is the case we have that $R\cdot R_{dN-2k(Y)}=0$ and by the surjectivity of $R_{dN-2k(Y)}\to R_{d}^\vee$ of Theorem \ref{macaulay} we conclude that $R=0$ in $R_d$, that is $R\in \sJ_d$.

The surjectivity of (\ref{surjok}) follows from the surjectivity at the level of the $H^0$:
\begin{equation*}
H^0(Y,\Omega_Y^N(dN-2N))\otimes H^0(Y,\Omega_Y^N(2N))\to H^0(Y,(\Omega_Y^{2N})(dN))
\end{equation*} which, similarly to the proof of Proposition \ref{surgettivita}, holds by the irreducibility of $H^0(Y,(\Omega_Y^{2N})(dN))$.
The case $Y=(D_3,\alpha_2)$, $d=3$ is excluded because in this case $H^0(Y,\Omega_Y^N(dN-2N))=0$.
\end{proof}

We conclude with our main theorem which relates the theory of infinitesimal Torelli theorem and Massey
products. This is Theorem \ref{A} from the Introduction.
\begin{thm}
\label{main}
Let $Y$ one of the spaces in Table \ref{tabellapesi} and $X=(\tau=0)$ a smooth hypersurface of degree $d\geq3$, except the cases $Y=Q^3$, $d=3$ and $Y=(D_3,\alpha_2)$, $d=3$. The following are equivalent
\begin{itemize}
	\item[\textit{i})] the differential of the period map $d\sP_X$ is zero on the infinitesimal deformation $\xi$ induced by $R\in H^0(Y,\sO_Y(d))$;
	\item[\textit{ii})] $R$ is an element of the pseudo-Jacobi ideal $\sJ_{d}$;
	\item[\textit{iii})]  for a generic $\xi$-Massey product $\Omega$ it holds $$\Omega\in \Ima H^0(X,\sO_X(2))\otimes \lambda^nW\to H^0(X,\Omega^{N-1}_X(2N));$$
	\item[\textit{iv})] if $\widetilde{\Omega}\in H^0(Y,\Omega_Y^N(2N))$ restricts to a generalized Massey product then $R\widetilde{\Omega}\in \sJ_{2N-k(Y)+d}$.
\end{itemize}
\end{thm}
%\subsection{Proof of the Main Theorem}
%
 %In this proof we denote as always by $\xi\in H^1(X,\Theta_X)$ the infinitesimal deformation and by $[R]\in R_{\sO_{\mathbb G}(a),\sigma}$ the corresponding element given by Proposition \ref{etutto}. We will use also Remark \ref{zerozero}. 
\begin{proof}
$i)\Leftrightarrow ii)$. Denote by $d\sP$ the differential of the period map. If $d\sP(\xi)=0$ then, by Theorem \ref{torelli}, $\xi=0$. By Proposition \ref{etutto} this means $R=0$ in $R_d$, that is $R\in \sJ_{d}$. The converse is trivial by the same Proposition \ref{etutto}.  
%
%$i)\Rightarrow iii)$. By Theorem [A] $\xi=0$.
%Now by  Lemma \ref{elenco} $(8)$ and by \cite[Proposition 3.1.6]{PZ} it follows that $D_W=0$ since $W$ is generic. Hence $iii)$ follows by the Viceversa of Theorem \ref{theoremA}.

%$iii)\Rightarrow i)$. Indeed Theorem \ref{theoremA} implies $$\xi\in\ker(H^1(X,(\Omega^1_{X}(2)^{\vee}\otimes_{\sO_X}\sO_X(2))\to 
 %H^1(X,(\Omega^1_{X}(2)^{\vee}\otimes_{\sO_X}\sO_X(2)\otimes_{\sO_X}\sO_X(D_W)$$ and by Lemma \ref{elenco} $(8)$ and by \cite[Proposition 3.1.6]{PZ} it follows that $D_W=0$ since $W$ is generic. Hence $\xi=0$.  
$ii) \Leftrightarrow iii)$ This is Theorem \ref{teoxi}.

 $ii)\Leftrightarrow iv)$ This is the content of Lemma \ref{quattro}. 

\end{proof}

\section{Hypersurfaces in log parallelizable varieties}
\label{par}
As we mentioned earlier, in \cite{RZ2} the case of smooth hypersurfaces in projective spaces is studied and then as generalization we have tackled the case of smooth hypersurfaces in Grassmannians in \cite{RZ3} and the case of smooth hypersurfaces in rational homogeneous varieties with Picard number one in this paper.
A different road that can be taken starting form the case of projective spaces is the study of smooth hypersurfaces in toric varieties and, more in general, of smooth hypersurfaces in  log parallelizable varieties.
This is what we will do now.
 
Consider a pair $(X,D)$ where $X$ is a smooth variety and $D$ is a reduced normal crossing divisor in $X$. That is, $D$ is an effective divisor locally given on an open set $U$ by $z_1z_2\cdots z_k=0$. We can define the sheaf $\Omega^1_X(\text{log }D)$ of \emph{logarithmic differentials} as the locally free $\sO_X$-module generated by $dz_1/z_1,\ldots,dz_k/z_k,dz_{k+1},\ldots,dz_{n}$ and its dual $T_X(-\text{log }D)$ is the subsheaf of the tangent sheaf $T_X$ consisting of those derivations that preserve the ideal sheaf $\sO_X(-D)$. This is also locally free and locally is generated by $z_1\frac{\partial}{\partial z_1},\dots,z_k\frac{\partial}{\partial z_k},\frac{\partial}{\partial z_{k+1}},\dots,\frac{\partial}{\partial z_n}$.

Now assume that a connected algebraic group $G$ acts on $X$ and preserves $D$, then from the natural morphism of Lie algebras 
$$
\lieg\to \Gamma(X,T_X)
$$
we have a map of sheaves
$$
\sO_X\otimes\lieg\to T_X
$$ which factors through a map
$$
\sO_X\otimes\lieg\to T_X(-\text{log }D).
$$
The variety $X$ is said \emph{log homogeneous under $G$ with boundary $D$} if this map is surjective and \emph{log parallelizable under $G$ with boundary $D$} if this map is an isomorphism.
Actually if $X$ is complete, it is log homogeneous (resp. log parallelizable) under some group $G$ with boundary $D$ if and only if the sheaf $T_X(-\text{log }D)$ is generated by its global sections (resp. is trivial). We then say that $X$ is log homogeneous (resp. log parallelizable) with boundary $D$ without specifying the group $G$. For more details on log homogeneous varieties we refer to \cite{br1}, \cite{br2}.

Here we deal with log parallelizable varieties and we recall some useful facts for later use. 
Take $X$ a smooth $n$-dimensional log parallelizable variety with boundary $D=\sum_{j=1}^r D_j$.
From the fact that $T_X(-\text{log }D)$ is trivial, we immediately have that 
$$
\Omega^n_X(\text{log }D)=\bigwedge^n\Omega^1_X(\text{log }D)=\Omega_X^n(D)
$$ is a trivial line bundle and therefore the canonical sheaf is $\Omega_X^n=\sO_X(-D)=\sO_X(-\sum_{j=1}^r D_j)$.
Furthermore we have the exact sequences given by the residue
\begin{equation}
0\to \Omega^1_X\to \Omega^1_X(\text{log }D)\to \bigoplus_j\sO_{D_j}\to0
\label{es1}
\end{equation}
\begin{equation}
0\to \Omega^a_X(\text{log }(D-D_1))\to \Omega_X^a(\text{log }D)\to \Omega^{a-1}_{D_1}(\text{log }(D-D_1)|_{D_1})\to 0
\label{es2}
\end{equation}
and
\begin{equation}
0\to \Omega^a_X(\text{log }D)(-D_1)\to \Omega_X^a(\text{log }(D-D_1))\to \Omega^{a}_{D_1}(\text{log }(D-D_1)|_{D_1})\to 0
\label{es3}
\end{equation}
which are true for every $X$ with reduced and normal crossing $D$ (see \cite[Properties 2.3]{es}) but in our case are considerably simpler because $\Omega^a_X(\text{log }D)$ is a direct sum of copies of $\sO_X$.

From these sequences we obtain
\begin{prop}
\label{in}
Let $X$ be an $n$-dimensional log parallelizable variety with boundary $D=\sum_{j=1}^r D_i$. The
non-empty partial intersections of boundary divisors,
$$
D_{k_1,\dots,k_m}:=D_{k_1}\cap\dots\cap D_{k_m}
$$ are log parallelizable varieties with boundary the restriction of 
$$
D^{k_1,\dots,k_m}:=\sum_{j\neq k_1,\dots,k_m}D_j
$$
\end{prop}
\begin{proof}
Consider the case of $D_1$ with boundary $D^1=D_2+\dots+D_r$. By Sequences (\ref{es2}) and (\ref{es3}) with $a=1$ we immediately have the commutative diagram 
\begin{equation}
\xymatrix{
&&0\ar[d]&0\ar[d]&\\
0\ar[r]&\bigoplus^n \sO_X(-D_1)\ar[r]\ar@{=}[d]&\Omega^1_X(\text{log }D^1)\ar[d]\ar[r]&\Omega^1_{D_1}(\text{log }D^1|_{D_1})\ar[r]\ar[d]&0\\
0\ar[r]&\bigoplus^n \sO_X(-D_1)\ar[r]&\bigoplus^n \sO_X\ar[d]\ar[r]&\bigoplus^n \sO_{D_1}\ar[r]\ar[d]&0\\
&&\sO_{D_1}\ar@{=}[r]\ar[d]&\sO_{D_1}\ar[d]&\\
&&0&0&
} \end{equation}
From this diagram it is clear that $\Omega^1_{D_1}(\text{log }D^1|_{D_1})\cong\bigoplus^{n-1}\sO_{D_1}$ is trivial.

The general statement for $D_{k_1,\dots,k_m}$ and boundary $D^{k_1,\dots,k_m}$ is obtaining by iteration of this argument.
\end{proof}
\begin{rmk}
Note that the canonical sheaf of $D_{k_1,\dots,k_m}$ is isomorphic to $\sO_{D_{k_1,\dots,k_m}}(-D^{k_1,\dots,k_m})$; this also comes from the adjunction formula.
\end{rmk}

We recall that toric varieties are an important subclass of log parallelizable varieties and some important cohomological results of toric varieties extend in fact to all log parallelizable varieties. The first one is the following
\begin{prop}
\label{van1}
Let $L$ be an ample invertible sheaf on a log parallelizable variety $X$. Then
$$
H^i(X,L)=0
$$ for all $i>0$.
\end{prop}
\begin{proof}
We work using induction on $n=\dim X$.
Consider the exact sequence 
\begin{equation}
0\to L(-D)\to L\to L|_D\to 0
\label{induz}
\end{equation}
Recall that the canonical sheaf of $X$ is $\sO_X(-D)$ hence $H^i(X,L(-D))$ vanishes for $i>0$ by Kodaira vanishing.

If $n=1$ it immediately follows that $H^1(X,L)=0$.

If $n>1$ we only have to prove the vanishings of $H^i(D,L|_D)$ which come from the induction hypothesis with some caution on the number of components of the divisor $D$. In fact if $D$ has only one irreducible component this vanishing is immediate by induction because by Proposition \ref{in} we have that $D$ is log parallelizable. Otherwise we write $D=D_1+D^1$ as before and by the exact sequence 
$$
0\to \sO_{D_1}(-D^1)\to\sO_D\to \sO_{D^1}\to 0
$$ tensored by $L$ we have 
$$
0\to L(-D^1)|_{D_1}\to L|_D\to L|_{D^1}\to 0.
$$
The vanishing of the higher cohomology of $L(-D^1)|_{D_1}$ comes again by Kodaira vanishing, on the other hand for $L|_{D^1}$ we can iterate the process and reduce again the number of components
$$
0\to L(-D^{1,2})|_{D_2}\to L|_{D^1}\to L|_{D^{1,2}}\to 0.
$$
Now note that the vanishing of the higher cohomology of $L(-D^{1,2})|_{D_2}$ is not immediate by Kodaira vanishing because $\sO_{D_2}(-D^{1,2})$ is not the canonical sheaf of $D_2$, in fact $\sO_{D_2}(-D^2)$ is. Nevertheless by the exact sequence
$$
0\to L(-D^{2})|_{D_2}\to L(-D^{1,2})|_{D_2}\to L(-D^{1,2})|_{D_{1,2}}\to 0
$$
we see that the higher cohomologies of $L(-D^{1,2})|_{D_2}$ vanish because those of $L(-D^{2})|_{D_2}$ and $L(-D^{1,2})|_{D_{1,2}}$ vanish by Kodaira on $D_2$ and $D_{1,2}$ respectively.
Hence it is not difficult to see that we can iterate this process until we obtain all the necessary vanishings.
\end{proof}

The second important result is the following theorem, which for toric varieties is called the vanishing theorem of Bott-Steenbrink-Danilov \cite[Theorem 7.1]{cox}
\begin{thm}
\label{bottst}
Let $L$ be an ample invertible sheaf on a log parallelizable variety $X$ with boundary $D$, then
$$
H^i(X,\Omega^j_X\otimes L)=0
$$ for every $j\geq0$ and $i>0$.
\end{thm}
\begin{proof}
Note that the vanishing of $H^i(X,\Omega^j_X(\textnormal{log }D)\otimes L)$ for $j\geq0$ and $i>0$ is immediate by Proposition \ref{van1} because $\Omega^j_X(\textnormal{log }D)\otimes L$ is just a direct sum of copies of $L$. Similarly the vanishing of $H^i(X,\Omega^j_X(\textnormal{log }D)\otimes L^{-1})$ for $j\geq0$ and $i<N$ follows by Kodaira vanishing theorem.

Now we use decreasing induction on the number of irreducible components of $D$ using Sequence \ref{es2} tensored by $L^{-1}$.
If we denote as before $D^1=D-D_1$, then the vanishing of 
$$H^i(X,\Omega^j_X(\textnormal{log }D^1)\otimes L^{-1})$$
for $j\geq0$ and $i<N$ follows from the vanishing of $H^{i-1}(D_1,\Omega^{j-1}_{D_1}(\textnormal{log }D^1)\otimes L^{-1}|_{D_1})$ and $H^i(X,\Omega^j_X(\textnormal{log }D)\otimes L^{-1})$ which hold by the previous step.

Continuing to remove components of $D$ at the end we obtain the vanishing of 
$$
H^i(X,\Omega^j_X\otimes L^{-1})
$$ $j\geq0$ and $i<N$ which by Serre duality gives the thesis.
\end{proof}
In the case of log-homogeneous varieties this vanishings hold for $i>j$, see \cite{br2}.

We end this section with an infinitesimal Torelli result for certain smooth hypersurfaces in log parallelizable varieties. The analogue for hypersurfaces in toric varieties is in \cite{Ikeda}.
 
Take an ample line bundle $L$ on $X$ and $Z$ a smooth element in the linear system $|L|$. We assume that $\Omega^n_X\otimes L$ is ample and that 
\begin{equation}
H^0(X,\Omega^n_X\otimes L)\otimes H^0(X,\Omega^n_X\otimes L^{n-1})\to H^0(X,{\Omega^n_X}\otimes\Omega^n_X\otimes L^n)
\label{multip}
\end{equation}
is surjective. This second hypothesis is strictly related to the following problem which is, at least for toric varieties, deeply studied:
\begin{prob}
For every $L_1$ and $L_2$ ample line bundles on a smooth projective toric variety $X$ the multiplication map  
\begin{equation}
H^0(X,L_1)\otimes H^0(X,L_{2})\to H^0(X,L_1\otimes L_2)
\end{equation} is surjective.
\end{prob}
Fakhruddin \cite{Fa} proved that this question has a positive answer for an ample line
bundle $L_1$ and a globally generated line bundle $L_2$ on a smooth toric surface.
 Moreover it is well known that for an ample line bundle L on a toric variety of dimension $n$, the multiplication map
$$
H^0(X,L^m)\otimes H^0(X,L)\to H^0(X,L^{m+1})
$$
is surjective for $m\geq n-1$. Hence our second hypothesis is very natural since we assume that $\Omega^n_X\otimes L$ is ample.

\begin{thm}
\label{torellilog}
Let $Z$ be a smooth element in the linear system $|L|$ with $L$ as above. Then the infinitesimal Torelli theorem holds for $Z$. 
\end{thm}
\begin{proof}
To prove that the infinitesimal Torelli holds for $Z$ it is enough to show that the map
\begin{equation}
\label{high}
H^1(X,T_Z)\to \text{Hom}(H^0(Z,\Omega^{n-1}_Z),H^1(Z,\Omega^{n-2}_Z))
\end{equation} is injective. In fact this map is the highest piece of the derivative of the period map, hence if (\ref{high}) is injective, the derivative of the period map is itself injective.

The idea is to prove the surjectivity of the dual of (\ref{high}), which is 
\begin{equation}
H^1(Z,\Omega^{n-2}_Z)^\vee\otimes H^0(Z,\Omega^{n-1}_Z)\to H^1(Z,T_Z)^\vee.
\label{highdual}
\end{equation} that is by Serre duality 
\begin{equation}
H^{n-2}(Z,\Omega^{1}_Z)\otimes H^0(Z,\Omega^{n-1}_Z)\to H^{n-2}(Z,\Omega^{n-1}_Z\otimes\Omega^1_Z)
\end{equation} 

First note that from the isomorphism $H^0(X,\Omega^n_X(Z))\cong H^0(X,\Omega^n_X(\textnormal{log }Z))$ composed with the residue map we obtain a map
\begin{equation*}
\phi\colon H^0(X,\Omega^n_X(Z))\cong H^0(X,\Omega^n_X(\textnormal{log }Z))\stackrel{Res}{\rightarrow} H^0(Z,\Omega^{n-1}_Z).
\end{equation*}

Now consider the exact sequence 
\begin{equation}
\label{seqlog}
0\to \Omega^2_X(\textnormal{log }Z)\to \Omega^2_X(Z)\to \Omega^3_X(2Z)|_Z\to \Omega^4_X(3Z)|_Z\to \dots\to \Omega^n_X((n-1)Z)|_Z\to 0
\end{equation}
This gives in cohomology a map
\begin{equation*}
H^0(Z,\Omega^n_X((n-1)Z)|_Z)\to H^{n-2}(X,\Omega^2_X(\textnormal{log }Z))
\end{equation*} which composed with restriction and the residue gives
\begin{equation*}
\phi'\colon H^0(X,\Omega^n_X((n-1)Z))\to H^0(Z,\Omega^n_X((n-1)Z)|_Z)\to H^{n-2}(X,\Omega^2_X(\textnormal{log }Z))\stackrel{Res}{\rightarrow} H^{n-2}(Z,\Omega^{1}_Z).
\end{equation*}

Taking the tensor product of Sequence (\ref{seqlog}) with the ample line bundle $\Omega^n_X\otimes L$ we obtain the exact sequence 
\begin{equation}
\begin{split}
0\to \Omega^2_X(\textnormal{log }Z)\otimes\Omega^n_X(Z)\to \Omega^2_X(Z)\otimes\Omega^n_X(Z)\to \Omega^3_X(2Z)|_Z\otimes\Omega^{n-1}_Z\to \dots\\\dots\to \Omega^n_X((n-1)Z)|_Z\otimes\Omega^{n-1}_Z\to 0.
\end{split}
\end{equation}
By Theorem \ref{bottst} and sequence 
$$
0\to \Omega_X^j((j-2)Z)\otimes\Omega^n_X(Z)\to \Omega_X^j((j-1)Z)\otimes\Omega^n_X(Z)\to \Omega_X^j((j-1)Z)|_Z\otimes\Omega^{n-1}_Z\to 0
$$ it is not difficult to see that $H^i(Z,\Omega_X^j((j-1)Z)|_Z\otimes\Omega^{n-1}_Z)=0$ for $i\geq1$ and since also $H^{n-1}(X,\Omega^2_X(Z)\otimes\Omega^n_X(Z))=0$ we have a surjective map
\begin{equation}
H^0(Z,\Omega^n_X((n-1)Z)|_Z\otimes\Omega^{n-1}_Z)\to H^{n-2}(X,\Omega^2_X(\textnormal{log }Z)\otimes\Omega^n_X(Z)).
\label{questa}
\end{equation}

By the exact sequence 
$$
0\to \Omega^2_X\otimes \Omega^n_X(Z)\to \Omega^2_X(\textnormal{log }Z)\otimes \Omega^n_X(Z)\to \Omega^1_Z\otimes \Omega^{n-1}_Z\to 0
$$ and the vanishing of $H^{n-1}(X,\Omega^2_X\otimes \Omega^n_X(Z))$ we have that the residue map 
$$
H^{n-2}(X,\Omega^2_X(\textnormal{log }Z)\otimes \Omega^n_X(Z))\to H^{n-2}(Z,\Omega^1_Z\otimes \Omega^{n-1}_Z)
$$ is also surjective.

Finally $H^1(X,\Omega^n_X\otimes\Omega^{n}_X((n-1)Z))=0$ and the restriction 
$$
H^0(X,\Omega^n_X((n-1)Z)\otimes\Omega^{n}_X(Z))\to H^0(Z,\Omega^n_X((n-1)Z)|_Z\otimes\Omega^{n-1}_Z)
$$ is also surjective.

Hence taking the composition of (\ref{questa}) together with the residue and the restriction we have a surjective map
$$
\psi\colon H^0(X,\Omega^n_X\otimes\Omega^{n}_X(nZ))\to H^{n-2}(Z,\Omega^1_Z\otimes \Omega^{n-1}_Z).
$$

We can write the following commutative diagram
\begin{equation}
\xymatrix{
H^0(X,\Omega^n_X(Z))\otimes H^0(X,\Omega^n_X((n-1)Z))\ar^-{\phi\otimes\phi'}[d]\ar[r]&H^0(X,\Omega^n_X\otimes\Omega^{n}_X(nZ))\ar^-{\psi}[d]\\
H^0(Z,\Omega^{n-1}_Z)\otimes H^{n-2}(Z,\Omega^{1}_Z)\ar[r]&H^{n-2}(Z,\Omega^1_Z\otimes \Omega^{n-1}_Z)
} \end{equation}
Since $\psi$ is surjective and the top row is surjective by our hypothesis on $L$ we are done.
\end{proof}

%%%%%%%%%%%%%%%%%%%%%%%%%%%%%%%%%%%%%%%%%%%%%%%%%%%%%%%
%%%%%%%%%%%%%%%%%%%%%%%%%%%%%%%%%%%%%%%%%%%%%%%%%%%%%%%
%%%%%%%%%%%%%%%%%%%%%%%%%%%%%%%%%%%%%%%%%%%%%%%%%%%%%%%


\begin{thebibliography}{Muk04}
%\bibitem[AC]{AC}
%E. Arbarello, M. Cornalba, \emph{Su una congettura di Petri}, Comment. Math. Helvetici 56 (1981), 1--38.
%
%
%\bibitem[BGN]{BGN}
%M. \'A. Barja, V. Gonz\'alez-Alonso, J. C. Naranjo, \emph{Xiao’s conjecture for general fibred surfaces}, Journal f\"ur die reine und angewandte Mathematik 739 (2018), 297--308. 
%
%
%\bibitem[BL]{BL}C. Birkenhake, H. Lange \emph{Complex Abelian Varieties} Grundlehren der
%mathematischen Wissenschaften 302, A Series of Comprehensive Studies in Mathematics, Springer-Verlag Berlin Heidelberg  (1992).
%C. Voisin, \emph{Hodge theory and complex algebraic geometry, I}. Translated from the French by Leila Schneps. Cambridge Studies in Advanced Mathematics, 76. Cambridge University Press, Cambridge, 2002.

\bibitem[BC]{cox}
V. Batyrev, D. Cox, \emph{On the Hodge structure of projective hypersurfaces in toric varieties}, Duke Math.
J. 75 (1994), 293--338.

\bibitem[B]{bott}
R. Bott, \emph{Homogeneous vector bundles}, Ann. of Math. 66 (1957), 203--248.
		
\bibitem[BR]{BR}
A. Borel, R. Remmert, \emph{\"Uber kompakte homogene K\"ahlersche Mannigfaltigkeiten},
Math. Ann. 145 (1962), 429--439.
		
%\bibitem[An]{andr} A. Andreotti, {\em  On a theorem of Torelli}, Amer. J. Math. 80 (1958), 801--828.
%
%
%\bibitem[BC]{BC} I. Bauer, F. Catanese, {\em Symmetry and variations of Hodge Structures}, Asian J. Math. 8 (2004), no. 2, 363--390. 
%

\bibitem[BGN]{BGN}
M. \'A. Barja, V. Gonz\'alez-Alonso, J. C. Naranjo, \emph{Xiao’s conjecture for general fibred surfaces}, Journal f\"ur die reine und angewandte Mathematik 739 (2018), 297--308. 
%\bibitem[Ca1]{Ca1}
%F. Catanese, {\em Infinitesimal Torelli theorems and counterexamples to
%Torelli problems}, Topics in transcendental algebraic geometry
%(Princeton, N.J., 1981/1982), 143--156, Ann. of Math. Stud., 106, Princeton Univ. Press,
%Princeton, NJ, (1984).
%

\bibitem[Br1]{br1}
M. Brion, \emph{Log homogeneous varieties}, Actas del XVI Coloquio Latinoamericano de \'Algebra, Revista Matem\'atica Iberoamericana, Madrid (2007), 1--39.

\bibitem[Br2]{br2}
M. Brion, \emph{Vanishing theorems for Dolbeault cohomology of log homogeneous varieties}, Tohoku Math. J. (2) 61 (2009), no. 3, 365--392.

%\bibitem[Ca2]{Ca2} F. Catanese, \emph{Moduli
%and classification of irregular Kaehler manifolds (and algebraic varieties) with Albanese
%general type fibrations},  Invent. Math. 104 (1991), no. 2, 263--289. 

%\bibitem[CD]{CD}
%Catanese, F., Dettweiler, M., \emph{Answer to a question by Fujita on Variation of Hodge Structures}, Higher Dimensional Algebraic Geometry: In honour of Professor Yujiro Kawamata's sixtieth birthday, 73--102, Mathematical Society of Japan, Tokyo, Japan, 2017.

%\bibitem[CCM]{CCM}
%F. Catanese, C. Ciliberto, M. Mendes Lopes,\emph{ On the 
%classification of irregular surfaces of general type with nonbirational bicanonical map},
%Trans. Amer. Math. Soc. 350 (1998), no. 1, 275--308. 
%
%\bibitem[CS]{CS} F. Catanese, F-O Schreyer,  \emph{Canonical 
%projections of irregular algebraic surfaces},
%Algebraic geometry, 79�116, de Gruyter, Berlin, (2002). 
%
%
%\bibitem[CMP]{CMP} J. Carlson, S. M{\"{u}}ller-Stach, C. Peters, 
%{\em Period Mappings and Period Domains}, Cambridge Studies in 
%Advanced Mathematics, 85. Cambridge University Press, Cambridge, (2003).
%

%\bibitem[CEZGT]{CEZGT} Eduardo Cattani, Fouad El Zein, Phillip A. Griffiths, and L\^e D\~ung Tr\'ang, editors.
%{\em Hodge theory}, volume 49 of Mathematical Notes. Princeton University Press, Princeton,
%NJ, 2014.

%\bibitem[CMR]{CMR}
% C. Ciliberto, M. Mendes Lopes, X. Roulleau \emph{ On Schoen surfaces},
%Comment. Math. Helv. 90 (2015), 59-74. 
%
%  
\bibitem[CNP]{CNP}
A. Collino, J. C. Naranjo, G. P. Pirola, {\em The Fano normal function}, J. Math. Pures Appl. (9) 98 (2012), no. 3, 346--366.
%
%\bibitem[Co]{cox}
%D. A. Cox, {\em Generic Torelli and infinitesimal variation of Hodge 
%structure}, Algebraic geometry, Bowdoin, 1985 (Brunswick, Maine, 1985), 235--246, Proc. 
%Sympos. Pure Math., 46, Part 2, Amer. Math. Soc., Providence, 
%RI, (1987).
%
%
%\bibitem[Cod]{Cod}
%G. Codogni, {\em Satake compactifications, Lattices and Schottky problem},  Ph.D. Thesis, (2013).
%

%
%\bibitem[C]{C}
%G. Ceresa, \emph{C is not algebraically equivalent to $C^-$ in its Jacobian}, Ann. of Math. (2) 117 (1983), 285--291.

\bibitem[CP]{CP} A. Collino, G. P. Pirola, 
\emph{The Griffiths infinitesimal invariant for a curve in its Jacobian}, Duke Math. J., 78 (1995), no. 1, 59--88.

%
%
%\bibitem[Do]{Do}
%R. Donagi, {\em Generic Torelli and variational Schottky}, Topics in transcendental algebraic geometry (Princeton, N.J., 1981/1982), 239--258,
%Ann. of Math. Stud., 106, Princeton Univ. Press, Princeton, NJ, (1984). 

%\bibitem[Fa]{Fa}
%N. Fakhruddin, \emph{Algebraic cycles on generic Abelian 
%varieties,} Comp. Math., 100, (1996), 101--119.


\bibitem[CRZ]{CRZ} L. Cesarano, L. Rizzi, F. Zucconi, \emph{On birationally trivial families and Adjoint quadrics}, submitted for publication, (2019).
%

\bibitem[CZ]{CZ} P. Corvaja, F. Zucconi, \emph{Bitangents to the quartic surface and infinitesimal deformations}, arXiv:1910.01365.

\bibitem[EV]{es}
H. Esnault, E. Viehweg, \emph{Lectures on Vanishing Theorems}, DMV Seminar 20, Birkh\"auser, Basel, 1992.


\bibitem[Fa]{Fa}
N. Fakhruddin, \emph{Multiplication maps of linear systems on smooth projective toric surfaces}, arXiv: math.
AG/0208178.

%\bibitem[Fu]{Fu}
%T. Fujita, \emph{On K\"ahler fiber spaces over curves}, 
 %J. Math. Soc. Japan 30 (1978), no. 4, 779--794. 


\bibitem[FH]{FH}
W. Fulton, J. Harris, \emph{Representation theory, a first course}, Springer 1991, GTM 133.


\bibitem[G-A]{victor}
V. Gonz\'alez-Alonso, \emph{On deformations of curves supported on rigid divisors}, Ann. Mat. Pura Appl. (4)
195(1), 111--132 (2016).
%\bibitem[G-A1]{victor1}
%V. Gonz\'alez-Alonso, \emph{Hodge numbers of irregular varieties and fibrations}, Ph.D. Thesis, (2013).
%
%\bibitem[Fl]{flenner}
%H. Flenner, {\em The infinitesimal Torelli problem for zero sets of sections of
%vector bundles},   Math. Z. 193 (1986), no. 2, 307--322. 
%
%\bibitem[G-A1]{victor1}
%V. Gonz\'alez-Alonso, \emph{Hodge numbers of irregular varieties and fibrations}, Ph.D. Thesis, (2013).
%
%\bibitem[G-A2]{victor2}
%V. Gonz\'alez-Alonso, \emph{On deformations of curves supported on rigid divisors}, arXiv:1401.7466, (2014).
%
%
%\bibitem[Gr1]{green1}
%M. L.  Green,  {\em  Koszul cohomology and the geometry of projective
%varieties}, J. Differential
%Geom. {19}  (1984), 125--171.

%\bibitem[Gr2]{green2}
%M. L.  Green,  {\em  Koszul cohomology and the geometry of projective
%varieties II}, J. Differential
%Geom. {20}  (1984), 279--289.

\bibitem[Gr1]{green1}
M. L.  Green,  {\em The period map for hypersurface sections of high
degree of an arbitrary variety}, 
Compositio Math. {55}  (1985), 135--156.

\bibitem[Gr2]{green2}
M. L. Green, \emph{Infinitesimal Methods in Hodge theory}, CIME Notes. Springer, 1994.


\bibitem[Griff]{Gri1}
P. Griffiths, {\em On the Periods of Certain Rational Integrals: I,II}, 
Ann. of Math. (2) 90 (1969), 460--495; ibid. (2) 90 (1969) 496--541. 

%\bibitem[Gri1]{grif1}
%P. Griffiths, {\em Periods of integrals on algebraic manifolds I,II}, 
%Amer. J. Math. {90} (1968), 568--626, 805--865.

%\bibitem[Gri2]{grif2}
%P. Griffiths, {\em Periods of integrals on algebraic manifolds III. Some global differential-geometric properties of the period mapping}, 
%Inst. Hautes \'Etudes Sci. Publ. Math. No. 38 (1970), 125--180.
%
%\bibitem[GS]{gr-s}
%P. Griffiths, W.  Schmid, {\em Recent developments in Hodge theory: a 
%discussion of techniques and results}, Discrete subgroups of Lie groups and applications to moduli 
%(Internat. Colloq., Bombay, 1973), pp. 31--127. Oxford Univ. Press, Bombay, (1975).

%
%
%\bibitem[GZ]{GZ} U. Garra, F. Zucconi, \emph{Very Ampleness and the Infinitesimal Torelli Problem},
%  Math. Z. 260 (2008), no. 1, 31--46. 
%
%
%\bibitem[KM]{K-M} K. Kodaira, J. Morrow, {\em Complex Manifolds},
%Holt-Rinehart-Winston, New York,
%(1971).
%
%\bibitem[Ku]{K} M. Kuranishi, {\em New proof for the existence of
%locally complete families of complex structures}, 
%1965 Proc. Conf. Complex Analysis (Minneapolis, 1964) pp. 142--154 Springer, Berlin.
% 
%\bibitem[Ii]{Ii} S. Itaka, 
%\emph{Algebraic Geometry}, Springer Verlag New York, Heidelberg, 
%(1982).
%
\bibitem[H]{Hump}
J. Humphreys, \emph{Introduction to Lie Algebras and Representation Theory}, Springer-Verlag,
Berlin, Heidelberg, New York, 1972.

\bibitem[Ik]{Ikeda}
A. Ikeda, \emph{Subvarieties of generic hypersurfaces in a nonsingular projective toric variety}, Mathematische Zeitschrift, 263 (4) (2008), 923--937.

%\bibitem[Il]{Il}
%L. Illusie, {\em R\'eduction semi-stable et d\'ecomposition de complexes de de Rham  \`a coefficients}, Duke Math. J. 60 (1990), no. 1, 139--185.
%\bibitem[LPW]{LPW}
%D. Lieberman, C. Peters, R. Wilsker, {\em A Theorem of Local-Torelli Type},
%Math. Ann.  231 (1977), 39--45.
%



%\bibitem[KO]{kob}
%Kobayashi, S., Ochiai, T., \emph{Meromorphic Mappings onto Compact Complex Spaces of General Type}, Inventiones mathematicae 31 (1976): 7--16. 
%\bibitem[Ko] {Ko}
%J. Koll\'ar, \emph{Families of varieties of general type}, 
%https://web.math.princeton.edu/~kollar/book/modbook20170720.pdf, (2017).


%\bibitem[K]{K}
%Konno, K., \emph{On the Irregularity of Special Non-Canonical Surfaces}, Publ. Res. Inst. Math. Sci., 30 (1994), no. 4,  671--688. 
\bibitem[K]{K}
K. Konno, \emph{Generic Torelli theorem for hypersurfaces of certain compact homogeneous K\"ahler manifolds}, Duke Math. J. 59 (1989), no. 1, 83--160.

\bibitem[K1]{K1}
K. Konno, \emph{Infinitesimal Torelli theorem for complete intersections in certain homogeneous K\"ahler manifolds}, Tohoku Math. J. (2) 38 (1986), no. 4, 609--624.

\bibitem[K2]{K2}
K. Konno, \emph{Infinitesimal Torelli theorem for complete intersections in certain homogeneous K\"ahler manifolds, II}, Tohoku Math. J. (2) 42 (1990), no. 3, 333--338. 

\bibitem[K3]{K3}
K. Konno, \emph{Infinitesimal Torelli theorem for complete intersections in certain homogeneous K\"ahler manifolds, III}, Tohoku Math. J. (2) 43 (1991), no. 4, 557--568.


\bibitem[Kos]{Kos}
B. Kostant, \emph{Lie algebra cohomology and the generalized Borel-Weil theorem}, Ann. of
Math. 74 (1961), 329--387.

%\bibitem[M] {M}
%M. M\"oller, \emph{Maximally irregularly fibred surfaces of general type}, Manuscripta Matematica Issue 1 vol. 116 (2005) 71--92.
%
%


%
%\bibitem[MR]{MR}
 %V. Maillot, D. R\"ossler {\em Une conjecture sur la torsion des classes de Chern des fibr\'es de Gauss-Manin}, Publ. Res. Inst. Math. Sci. 46 (2010), no. 4, 789--828.

%
%\bibitem[Mo]{Mo} S. Mori, \emph{Classification of higher-dimensional varieties Proc}, Symp. Pure Math. 46 (1987), 269--332.
%\bibitem[N] {N}
%M.V. Nori, \emph{Algebraic Cycles and Hodge Theoretic Connectivity}, Inv. Math., 111
%(1993), 349--373.
%
%\bibitem[OS]{OS} F. Oort, J. Steenbrink, 
%{\em The local Torelli problem for algebraic curves}, 
 %Journ\'ees de G\'eom\'etrie Alg\'ebrique d'Angers, Juillet 1979/Algebraic Geometry, Angers, 1979,
 %pp. 157--204, Sijthoff and Noordhoff, Alphen aan den Rijn--Germantown, Md., (1980). 
%

\bibitem[O]{O}
G. Ottaviani, \emph{Rational Homogeneous Varieties}, notes available at \url{http://web.math.unifi.it/users/ottaviani/rathomo/rathomo.pdf}


%
%\bibitem[PP]{pp}
%Giuseppe Pareschi and Mihnea Popa, \emph{Strong generic vanishing and
%a higher-dimensional Castelnuovo-de Franchis inequality}, Duke Math.
%J., 150(2):269--285, 2009.
%\bibitem[Pe]{Pe} C. Peters, {\em Some Remarks About Reider's 
%Article ``Infinitesimal Torelli Theorem for certain irregular 
%surfaces of general type}, Math. Ann.  281 (1988), 315--324.

\bibitem[PT]{PT}
G.P. Pirola, S. Torelli, \emph{Massey Products and Fujita decompositions on fibrations of curves}, Collectanea Mathematica, 71 (2020), 39--61. 

%\bibitem[PS]{PS} C. A. M. Peters, J. H. M. Steenbrink, {\em Mixed Hodge structures}, Volume 52
%of Ergebnisse der Mathematik und ihrer Grenzgebiete. 3. Folge. A Series of Modern
%Surveys in Mathematics. Springer--Verlag, Berlin, 2008.

%
\bibitem[PR]{PR}
G. P. Pirola, C. Rizzi, {\em Infinitesimal invariant and vector bundles}, Nagoya Math. J., 186 (2007), 95--118.

%\bibitem[P]{P}
%G. P. Pirola, \emph{On a conjecture of Xiao,} J. Reine Angew. Math., 431 (1992), 75--89.
%
\bibitem[PZ]{PZ}
G. P. Pirola, F. Zucconi, \emph{Variations of the Albanese morphisms,} J. Algebraic Geom., 12 (2003), no. 3, 535--572. 

%
%\bibitem[Ran]{Ran}
%Z. Ran \emph{On subvarieties of abelian varieties}, Invent. Math., 62(3) 459--479, 1981.


\bibitem[Ra]{Ra}
E. Raviolo, {\em Some geometric applications of the theory of variations of Hodge structures}, Ph.D. Thesis.
%
%\bibitem[Re]{Re} I. Reider, {\em On the Infinitesimal Torelli Theorem for certain irregular surfaces of general type}, Math. Ann.  280 (1988), 285--302.
%
%\bibitem[Re2]{Re2} 
%I. Reider, \emph{Nonabelian Jacobian of projective surfaces. Geometry and representation theory}, Lecture Notes in Mathematics, 2072. Springer, Heidelberg, 2013.

\bibitem[RZ1]{RZ1} L. Rizzi, F. Zucconi, {\em Differential forms and quadrics of the canonical image}, Annali di Matematica  Pura ed Applicata, Vol. 199, Issue 6 (2020), 2341--2356.

 
\bibitem[RZ2]{RZ2} L. Rizzi, F. Zucconi, {\em Generalized adjoint forms on algebraic varieties,}  Annali di Matematica Pura ed Applicata, Vol. 196, Issue 3 (2017), 819--836.
\bibitem[RZ3]{RZ3} L. Rizzi, F. Zucconi, {\em On Green’s proof of infinitesimal Torelli theorem for hypersurfaces,} Rendiconti Lincei -- Matematica e Applicazioni,  Vol. 29, Issue 4 (2018), 689--709.
%
\bibitem[RZ4]{RZ4} L. Rizzi, F. Zucconi, {\em Fujita decomposition and Massey product for fibered varieties}  \url{https://arxiv.org/abs/2007.01473}, (2020).
%
\bibitem[Sn]{snow}
D. Snow, \emph{Cohomology of twisted holomorphic forms on Grassmann manifolds and quadric hypersurfaces},
Math. Ann. 276 (1986), no. 1, 159--176. 

%\bibitem[To]{to}
%R. Torelli, {\em Sulle variet\`{a} di Jacobi, I, II}, Rendiconti R. 
%Accad. dei Lincei {22-2}  (1913), 98--103, 437--441.
%
%\bibitem[To]{To} S. Torelli, {\em Fujita decompositions and infinitesimal
%invariants on fibred surfaces}, 2018, PhD Thesis.

%\bibitem[Vo]{Vo} C. Voisin, {\it Une Remarque Sur l'Invariant Infinit\'esimal 
%Des Fonctions Normales,} C. R. Acad. Sci. Paris, t. 307, S\'erie I, 
%(1988), 157--160.
%

%\bibitem[U]{U}
%K. Ueno, \emph{Classification theory of algebraic varieties and compact complex spaces}, Lecture
%Notes in Math., vol. 439, Springer-Verlag, 1975.


%\bibitem[Vo1]{Vo1}
%C. Voisin, \emph{Hodge theory and complex algebraic geometry, I}. Translated from the French by Leila Schneps. Cambridge Studies in Advanced Mathematics, 76. Cambridge University Press, Cambridge, 2002.
%
%\bibitem[Vo2]{Vo2}
%C. Voisin, \emph{Hodge theory and complex algebraic geometry, II}. Translated from the French by Leila Schneps. Cambridge Studies in Advanced Mathematics, 77. Cambridge University Press, Cambridge, 2003.

%\bibitem[We]{we}
%A. Weil, {\em
%Zum Beweis des Torellischen Satzes}, Nachr. Akad. Wiss. G\"ottingen, 
%Math.-Phys. Kl.   IIa (1957), 32--53.
%\bibitem[X]{X}
%G. Xiao, \emph{Fibered Algebraic Surfaces with Low Slope,} Math. Ann. 276,  (1987), 449--466.


%\bibitem[X]{X}
%Xiao, G., \emph{Algebraic surfaces with high canonical degree}, Math. Ann., 274 (1986), 473--483.

\end{thebibliography}
\end{document}